\date{\today}
\documentclass{amsproc}

\usepackage{amssymb}

\usepackage{graphicx}

\usepackage[cmtip,all]{xy}

\usepackage[all]{xy}
\usepackage{amsthm}
\usepackage{amscd}
\usepackage{amsmath}
\usepackage{amsfonts}
\usepackage{amssymb}
\usepackage{stmaryrd}
\usepackage{graphicx}
\usepackage{color}
\usepackage{hyperref}
\hypersetup{
    bookmarks=true,         
    unicode=false,          
    pdftoolbar=true,        
    pdfmenubar=true,        
    pdffitwindow=false,     
    pdfstartview={FitH},    
    pdftitle={My title},    
    pdfauthor={Author},     
    pdfsubject={Subject},   
    pdfcreator={Creator},   
    pdfproducer={Producer}, 
    pdfkeywords={keyword1, key2, key3}, 
    pdfnewwindow=true,      
    colorlinks=false,       
    linkcolor=red,          
    citecolor=green,        
    filecolor=magenta,      
    urlcolor=cyan           
}

\newtheorem{theorem}{Theorem}[section]
\newtheorem{lemma}[theorem]{Lemma}

\theoremstyle{definition}
\newtheorem{definition}[theorem]{Definition}
\newtheorem{example}[theorem]{Example}
\newtheorem{proposition}[theorem]{Proposition}
\newtheorem{corollary}[theorem]{Corollary}

\theoremstyle{remark}
\newtheorem{remark}[theorem]{Remark}

\numberwithin{equation}{section}

\DeclareMathOperator {\Map}{Map}

\def \Z {\mathbb{Z}}
\def \C {\mathbb{C}}
\def \R {\mathbb{R}}
\def \O {\mathcal{O}}
\def \T {\mathbb{T}}

\newcommand{\git}{/\!\!/}

\newcommand{\pt}{\text{\textnormal{pt}}}

\begin{document}
\title{Twisted equivariant quasi-elliptic cohomology and M-brane charge} 

\author{Zhen Huan}

\address{Zhen Huan, Center for Mathematical Sciences,
Huazhong University of Science and Technology, Hubei 430074, China} \curraddr{}
\email{2019010151@hust.edu.cn}

\date{\today}


\begin{abstract}    

In this paper we construct a twisted version of quasi-elliptic cohomology \cite{huan2018}. This theory can be constructed as a K-theory of formal loop spaces. After establishing basic properties of the theory, including restriction, change-of-group and induction maps, we construct the Chern character map. 
And we compute the twisted quasi-elliptic cohomology theories of representation 4-spheres acted on by the finite subgroups of $SU(2)$, which, as conjectured in \cite{SatiSchreiber2022}, can produce good observables on M-brane charge in a Tate-elliptic enhancement of D-brane charge in twisted equivariant K-theory.

\end{abstract}

\maketitle

\section{Introduction}

    In this paper we construct a twisted version of elliptic cohomology. Quasi-elliptic cohomology, introduced by the author in \cite{huan2018}, is a variant of Tate K-theory, which is the generalized elliptic cohomology associated to
the Tate curve. The Tate curve $Tate(q)$ is a generalized  elliptic curve
over Spec$\mathbb{Z}((q))$, which is classified as the completion
of the algebraic stack of some nice generalized elliptic curves at
infinity 
\cite[Section 2.6]{ando2001}. The relation between quasi-elliptic cohomology and Tate K-theory can be expressed by \begin{equation}QEll^*(X/\!\!/G)\otimes_{\mathbb{Z}[q^{\pm}]}\mathbb{Z}((q))\cong K^*_{Tate}(X/\!\!/G).\label{relqelltatek}\end{equation}Quasi-elliptic cohomology is not an elliptic cohomology but it contains all the information of equivariant
Tate K-theory. That’s how it got its name.
In addition, it can be constructed as the orbifold K-theory of a loop groupoid $\Lambda(X/\!\!/G)$, which partially proved a conjecture by Witten \cite{landweber1988} emphasizing the relation between elliptic cohomology and circle-equivariant K-theory of a free loop space. 

One classical example of twisted cohomology theories is twisted K-theory \cite{karoubi1970b, atiyahsegal2005}, which admits a geometric construction. The relation between twisted K-theory and physics has been observed for decades. It is conjectured in \cite{Maldacena2001DBraneIA} that it can classify D-branes, Ramond-Ramond field strengths and  spinors in type II string theory under some conditions. Since elliptic cohomology is a higher version of K-theory, it is natural to expect a relation between a twisted version of elliptic cohomology and physics. 

Let $G$ be a finite group acting on a space $X$. The twisting that we use to construct twisted quasi-elliptic cohomology $QEll^{\alpha + \ast}(-)$ corresponds to an element $\alpha \in H^3(BG; U(1))$. Motivated by the relation between elliptic cohomology and loop groupoid, we construct  $QEll^{\alpha + \ast}(-)$ in Section \ref{twistedqell}  as the orbifold K-theory of 
a twisted version of the loop groupoid $\Lambda(X/\!\!/ G)$, which embodies both the loop rotation and  the central extension classified by the 2-cocycle obtained from the loop transgression of the twist $\alpha$.
 We  prove that $QEll^{\alpha + \ast}(-)$ has many parallels with twisted equivariant $K$-theory, including the existence of K\"{u}nneth maps, induction maps, change-of-group isomorphisms and Chern characters. 
It has the relation with the twisted equivariant Tate K-theory $K^{\alpha + \ast}_{Tate}(-/\!\!/G)$ defined in \cite{dove2019} as below.
\[QEll^{ \alpha + \ast }_G(X)\otimes_{\mathbb{Z}[q^{\pm}]}\mathbb{Z}((q))\cong K^{ \alpha + \ast}_{Tate}(X/\!\!/G)\]

The author gives a loop space construction of quasi-elliptic cohomology via bibundles in
\cite[Section 2]{huan2018}. We recall the model $Loop(X/\!\!/G)$ in  Section \ref{qellloop}, which is constructed from the category of bibundles from $S^1/\!\!/\ast$ to $X/\!\!/G$ enriched by the rotation of loops. In Section \ref{twistedloopspace}, we construct a groupoid $Loop^{twist}(X/\!\!/ G)$ of twisted equivariant loops. A subgroupoid of it consisting of constant loops in $Loop^{twist}(X/\!\!/ G)$  provides a loop space model for twisted quasi-elliptic cohomology, thus, for twisted equivariant Tate K-theory as well.

Sati and Schreiber conjecture in \cite{SatiSchreiber2022} that the twisted quasi-elliptic cohomology of representation 4-spheres of finite subgroups of $SU(2)$ can produce good observables on M-brane charges in a Tate-elliptic enhancement of D-brane charge in twisted equivariant K-theory. In Section \ref{ex:qell:s4} and \ref{ex:twisted:qell:s4}, we show the study of this conjecture is feasible mathematically by computing the twisted quasi-elliptic cohomology theories of representation 4-spheres acted on by any finite subgroup of $SU(2)$. 
The theories are computed by applying the properties of quasi-elliptic cohomology theories and  equivariant K-theories, especially the conclusions in Appendix \ref{Cor:decomp:dcl}, which are corollaries of the decomposition formula in \cite[Theorem 3.6 and Corollary 3.7]{ngel2017EquivariantCB}.
As these computations show, twisted quasi-elliptic cohomology groups appear like twisted equivariant K-theory groups as expected for D-branes on orbifolds in type II string theory but subject to certain curious adjustment coming from a circle action, much as expected for M-theoretic corrections under its duality to type IIB under a circle fibration. This is plausibly indicative of the suggested relation to M-brane charge that was mentioned in the introduction. Further discussion of this point is however beyond the scope of this article.

In Section \ref{qell} we give a sketch of quasi-elliptic cohomology, including its definition, basic properties, and the loop space construction. In Section \ref{devotoequiell}, we review Devoto's equivariant elliptic cohomology. In Section \ref{twistedellold}, we recall the definition of twisted equivariant elliptic cohomology. In Section \ref{twistedqell}, we construct twisted quasi-elliptic cohomology. In Section \ref{twistedloopspace}, we define a model of twisted loop space, with which we can construct twisted quasi-elliptic cohomology. In Section \ref{twistedChernS}, based on the Chern character of quasi-elliptic cohomology, 
we construct a Chern character map of twisted quasi-elliptic cohomology. 
In Section \ref{compute:twisted:qell}, we compute more examples of quasi-elliptic cohomology and twisted quasi-elliptic cohomology. Especially in Section \ref{ex:qell:s4} and \ref{ex:twisted:qell:s4}, we compute the twisted quasi-elliptic cohomology of representation 4-spheres acted on by all the finite subgroups of $SU(2)$.

\section*{Acknowledgments} 

This research is based upon the work  supported by the Young Scientists Fund of the National Natural Science Foundation of China (Grant No. 11901591), the General Program of the National Natural Science Foundation of China (Grant No. 12371068)  for the project 
"Quasi-elliptic cohomology and its application in topology and mathematical physics", the National Science Foundation under Grant Number DMS 1641020,  and the research funding from Huazhong University of Science and Technology.

The research was inspired during the Mathematical Research Community (MRC) workshop at Rhode Island in June 2019. The author feels indebted to Dileep Menon, Matthew Bruce Young for valuable supports and advices, especially the organisers Nora Ganter and Daniel Berwick-Evans.
The author thanks Hisham Sati and Urs Schreiber for suggesting the author compute twisted quasi-elliptic cohomology of spheres, and thanks Center for Quantum and Topological Systems at New York University Abu Dhabi for hospitality and support. In addition, the author thanks the Max Planck Institute for Mathematics for hospitality and support. And the author thanks the helpful comments from the referees.

This project was originally intended as a collaboration with M. Spong, and the author is thankful for his insight and many deep conversations about the subject.


\section{Quasi-elliptic cohomology}\label{qell}

In this section we review quasi-elliptic cohomology, the main reference of which is \cite{huan2018}. 
 It is a variant of Tate
K-theory \cite{ando2001} \cite{ganter2014}. Many constructions in it can be made neater than Tate K-theory and most elliptic cohomology theories. As shown in Section
\ref{qellloop}, it can be constructed as orbifold K-theory of a loop space.

\subsection{Definition} \label{qelldef}
In this section we recall the definition of quasi-elliptic
cohomology in term of equivariant K-theory and state the
conclusions that we need in this paper. For more details on quasi-elliptic
cohomology, please refer to \cite{huan2018}. 

Let $G$ be a  group and  $X$ a $G$-space. Let $G^{tors}\subseteq G$ denote the set of
torsion elements of $G$. For any $g\in G^{tors}$, the fixed point
space $X^{g}$ is a $C_G(g)-$space where $C_G(g)$ is the centralizer $\{h\in G\mid h g= g h\}$. This group action can be extended to that by the group
\[\Lambda_G(g):=C_G(g)\times \mathbb{R}/\langle(g, -1)\rangle,\] which is given explicitly by 
\[[h, t]\cdot x:=h \cdot x,\] for any $[h, t] \in \Lambda_G(g)$ and $x\in X^{g}$. Let $\mathbb{T}$ denote the circle group $\mathbb{R}/\mathbb{Z}$. 

The inertia orbifold  of $X\git G$ is defined to be the functor groupoid \[ \hom_{grpd}(B\Z, X\git G). \]  The
inertia orbifold plays an import role in the geometry and index theory of orbifolds \cite{Adem_Leida_Ruan_2007}\cite{PFLAUM200783}. The rotation action of the circle group $\mathbb{T}$ can be built into it as new morphisms, as we see in Definition \ref{huan:inertia:extended}.

To give a complete description of the extended inertia orbifold $\Lambda(X/\!\!/G)$, we need the following definitions.

\begin{definition} 
\begin{enumerate}\item Let $g$, $g'$ be two elements in $G$. Define $C_G(g, g')$ to be the set 
$\{ h\in G\mid g'h=hg\}.$

\item Let $\Lambda_G(g, g')$ denote
the quotient of $C_G(g, g')\times \mathbb{R}/l\mathbb{Z}$ under
the equivalence $$(\alpha, t)\sim (g'\alpha, t-1)=(\alpha g,
t-1),$$ where $l$ is the order of $g$ in $G$. \end{enumerate} \end{definition}

\begin{definition} Define $\Lambda(X/\!\!/G)$ to be the groupoid with
\begin{itemize}
\item \textbf{objects}: the space $\coprod\limits_{g\in G^{tors}}X^{g}$
\item 
\textbf{morphisms}: the space
$$\coprod\limits_{g, g'\in G^{tors}}\Lambda_G(g, g')\times X^g.$$ \end{itemize}
For an object $x\in X^g$, the morphism
$([\alpha, t], x)\in \Lambda_G(g, g')\times X^g$ is an arrow from $x$ to $\alpha \cdot x\in X^{g'}.$ The composition of the morphisms is
defined by
\begin{equation}([\alpha_1, t_1], \alpha_2\cdot x) \circ ([\alpha_2,
t_2], x) = ([\alpha_1\alpha_2, t_1+t_2], x).\end{equation} 
We have a
homomorphism of orbifolds
$$\pi: \Lambda(X/\!\!/G)\longrightarrow B\mathbb{T}$$ sending all the objects to the single object in $B\mathbb{T}$, and
a morphism $([\alpha,t], x)$ to $e^{2\pi it}$ in $\mathbb{T}$.
\label{huan:inertia:extended}
\end{definition}

     An intrinsic version of Definition \ref{huan:inertia:extended} is provided in \cite[Theorem 2.3]{SatiSchreiber2022}.

\begin{definition} The quasi-elliptic cohomology $QEll^*_G(X)$ is defined to be
$K^*_{orb}(\Lambda(X/\!\!/G))$. \label{defqell1}\end{definition}

The groupoid $\Lambda(X/\!\!/G)$ is equivalent to the disjoint union of action groupoids \[\coprod_{g\in G^{tors}_{conj}} X^g \git \Lambda_G(g) \] where $G^{tors}_{conj}$ is the set of a family of representatives of the $G$-conjugacy classes in $G^{tors}$. Thus,
we can unravel Definition \ref{defqell1} and
express it via equivariant K-theory.


\begin{definition}
\begin{equation}\label{defqell} QEll^*_G(X):=\prod_{g\in
G^{tors}_{conj}}K^*_{\Lambda_G(g)}(X^{g})=\bigg(\prod_{g\in
G^{tors}}K^*_{\Lambda_G(g)}(X^{g})\bigg)^G,\end{equation} where
$G^{tors}_{conj}$ is a set of representatives of $G-$conjugacy
classes in $G^{tors}$.
\end{definition}

\begin{example}
When the group $G$ is the trivial group $e$, by definition, the quasi-elliptic cohomology $QEll^*_{e}(X)$ has only one factor and it is equivalent to $\mathbb{T}$-equivariant K-theory.
\end{example}

As indicated in \cite[Section 5]{Kitchloo_Morava_2007} and \cite[Section 5.2]{lurie2009b}, Tate K-theory can be expressed as the completion of $\mathbb{T}$-equivariant K-theory of free loop spaces. Explicitly, \[ K^*_{Tate}(X) \cong K_{\mathbb{T}}^*(X)\otimes_{\mathbb{Z}[q^{\pm}]} \mathbb{Z}((q)),\] 
where $X$ is a space with trivial $\mathbb{T}$-action.
For equivariant Tate K-theory, we have a further conclusion generalizing this relation.

Consider the composition 
$$\mathbb{Z}[q^{\pm}]=K^0_{\mathbb{T}}(\mbox{pt})\buildrel{\pi^*}\over\longrightarrow K^0_{\Lambda_G(g)}(\mbox{pt})\longrightarrow
K^0_{\Lambda_G(g)}(X)$$ where $\pi: \Lambda_G(g)\longrightarrow
\mathbb{T}$ is the projection $[a, t]\mapsto e^{2\pi i t}$ and the
second map is defined via the collapsing map $X\longrightarrow \mbox{pt}$. Via it, 
$QEll_G^*(X)$ is naturally a
$\mathbb{Z}[q^{\pm}]-$algebra. 

\begin{proposition}The relation between quasi-elliptic cohomology and equivariant 
Tate K-theory $K^*_{Tate}(- \git G)$ is
\begin{equation}QEll^*_G(X) \otimes_{\mathbb{Z}[q^{\pm}]}\mathbb{Z}((q)) \cong K^*_{Tate}(X\git G).
\label{tateqellequiv}\end{equation}\end{proposition} This is the main reason why the theory is called quasi-elliptic cohomology.
Proposition \ref{tateqellequiv} is given in \cite[(1.1)]{huan2018} and also \cite[Remark 6.19]{dove2019}.

In addition, we give an example computing quasi-elliptic cohomology, which is \cite[Example 3.3]{huan2018}. The conclusions in Example \ref{ex3.3:huan2018} are applied in Section \ref{compute:twisted:qell}.
\begin{example}[$G=\mathbb{Z}/N$] \label{ex3.3:huan2018}
Let $G=\mathbb{Z}/N$ for $N\geq 1$, and let $\sigma\in
G$. Given an integer $k\in\mathbb{Z}$ which projects to
$\sigma\in\mathbb{Z}/N$, let $x_k$ denote the
representation of $\Lambda_G(\sigma)$ defined by
\begin{equation}\begin{CD}\Lambda_{G}(\sigma)=(\mathbb{Z}\times\mathbb{R})/(\mathbb{Z}(N,0)+\mathbb{Z}(k,1))
@>{[a,t]\mapsto[(kt-a)/N]}>> \mathbb{R}/\mathbb{Z}=\mathbb{T}
@>{q}>> U(1).\end{CD}\label{xk}\end{equation} $R\Lambda_G(\sigma)$
is isomorphic to the ring $\mathbb{Z}[q^{\pm}, x_k]/(x^N_k-q^k)$.
\bigskip

For any finite abelian group 
$G=\mathbb{Z}/N_1\times\mathbb{Z}/N_2\times\cdots\times\mathbb{Z}/N_m$,
let $\sigma=(k_1, k_2, \cdots k_n)\in G.$ We have
$$\Lambda_G(\sigma)\cong\Lambda_{\mathbb{Z}/N_1}(k_1)\times_{\mathbb{T}}\cdots\times_{\mathbb{T}}\Lambda_{\mathbb{Z}/N_m}(k_m).$$ Then,
the representation ring of $\Lambda_G(\sigma)$
\begin{align*}R\Lambda_G(\sigma)&\cong
R\Lambda_{\mathbb{Z}/N_1}(k_1)\otimes_{\mathbb{Z}[q^{\pm}]}\cdots\otimes_{\mathbb{Z}[q^{\pm}]}R\Lambda_{\mathbb{Z}/N_m}(k_m)\\
&\cong\mathbb{Z}[q^{\pm}, x_{k_1}, x_{k_2},\cdots
x_{k_m}]/(x^{N_1}_{k_1}-q^{k_1},x^{N_2}_{k_2}-q^{k_2}, \cdots
x^{N_m}_{k_m}-q^{k_m})\end{align*} where all the $x_{k_j}$'s are
defined as $x_k$ in (\ref{xk}).\label{ppex}
\end{example}

\subsection{Loop space}\label{qellloop}

The author provides a loop space construction for quasi-elliptic cohomology in \cite[Section 2]{huan2018}. We review that model in this section.

The classical loop space is on the level $0$: we start with a smooth manifold $X$ and define a space  of free loops
\begin{equation}LX:=\mathbb{C}^{\infty}(S^1, X). \end{equation} It
comes with an evident action by the circle group
$\mathbb{T}$ defined by rotating the circle
\begin{equation}t\cdot \gamma:= (s\mapsto \gamma (s+t)), \mbox{
}t\in \T, \mbox{   } \gamma\in LX.\end{equation}

Let $G$ be a finite group acting on  $X$  from the right. The free loop space $LX$  is  equipped with an action
by the loop group $LG$   \begin{equation}\delta\cdot
\gamma:=(s\mapsto \delta(s)\cdot \gamma(s)),\mbox{ for any } s\in
 S^1, \mbox{   }\delta\in LX, \mbox{ }\mbox{
    }\gamma\in LG.\end{equation}

Combining the action by the group of automorphisms $Aut(S^1)$ on the
circle and the action by $LG$, we get an action by the extended
loop group $\Lambda G$ on $LX$, where \[\Lambda G:=LG\rtimes\mathbb{T}\]
is a subgroup of
\begin{equation} LG\rtimes Aut(S^1), \mbox{             }(\gamma, \phi)\cdot (\gamma', \phi'):= (s\mapsto \gamma(s) \gamma'(\phi^{-1}(s)), \phi\circ\phi')\end{equation}
with $\mathbb{T}$ identified with the group of rotations on $S^1$.
The group $\Lambda G$ acts on $LX$ by
\begin{equation}\delta \cdot(\gamma, \phi):= (t\mapsto
\delta(\phi(t))\cdot\gamma(\phi(t))), \mbox{  for any }(\gamma,
\phi)\in \Lambda G,\mbox { and    }\delta \in
LX.\label{loop2action}\end{equation}
Let $\widetilde{\delta}: G\times \mathbb{T}\longrightarrow X$ denote the map $(g, t)\mapsto \delta(t)g$.
The action on $\delta$ by $(\gamma, t)$ can be interpreted as precomposing $\widetilde{\delta}$ with a $G-$bundle map covering the rotation $\phi$.
\begin{equation}\xymatrix{
G\times \mathbb{T} \ar[rrr]^{(g, t)\mapsto (\gamma(t)g, \phi(t)) } \ar[d] &&&G\times \mathbb{T}\ar[d]\ar[r]^{\widetilde{\delta}} &X\\
\mathbb{T}\ar[rrr]^{\phi} &&&\mathbb{T}&}\label{lgtact}\end{equation}
More generally, we have the definition of the equivariant loop space $Loop(X/\!\!/G)$ below.

\begin{definition}We define the equivariant loop space $Loop(X/\!\!/G)$ as the category
with objects \[\xymatrix{
\mathbb{T} &P\ar[l]_>>>>>{\pi} \ar[r]^{f} &X}\]
where $\pi $ is a principal $G-$bundle over $\mathbb{T}$
and $f$ is a $G-$map. A morphism \[(\alpha, t): \{\xymatrix{
\mathbb{T} &P'\ar[l]_>>>>>{\pi} \ar[r]^{f'} &X}\}\longrightarrow \{\xymatrix{
\mathbb{T} &P\ar[l]_>>>>>{\pi} \ar[r]^{f} &X}\}\] consists of a $G-$bundle map $\alpha$ and a rotation $t$ making the diagrams commute. \[\xymatrix{
P'\ar@/^1pc/[rr]|-{f'} \ar[r]|-{\alpha} \ar[d] &P\ar[d]\ar[r]|-{f}&X\\
\mathbb{T}\ar[r]|-{t} &\mathbb{T}&}\] \label{catequivloop}
\end{definition}
In this way, starting from a groupoid $X\git G$, we get a loop groupoid $Loop(X/\!\!/G)$.

\begin{remark}The category $Loop(X/\!\!/G)$ has the same objects as the category of bibundles $Bibun(\mathbb{T}/\!\!/\ast, X/\!\!/G)$. The morphisms in $Bibun(\mathbb{T}/\!\!/\ast, X/\!\!/G)$ are those of the form $(\alpha, 0)$ in $Loop(X/\!\!/G)$.  \end{remark}

\begin{proposition}The groupoid $\Lambda(X/\!\!/G)$ is a subgroupoid of $Loop(X/\!\!/G)$ consisting of constant loops. \end{proposition}

\section{Devoto's equivariant elliptic cohomology over $\mathbb{C}$}\label{devotoequiell}

For the rest part of the paper, the equivariance group $G$ is always a finite group, unless otherwise specified.

In \cite{devoto1996} Devoto provided a $G$-equivariant refinement of the elliptic cohomology defined by Landweber, Ravenel and Stong in \cite{LRS2004} for finite groups $G$. In this section, we give a brief overview of his construction. Another reference for this section is \cite[Section 3]{Berwick2022}. We construct in Section \ref{QEllChernS}  the Chern character from quasi-elliptic cohomology to a variant of Devoto's equivariant elliptic cohomology.

Let $\mathcal{C}(G)$ denote the set of pairs of commuting elements of $G$, and $L \subset \mathbb{C}^2$ the subspace of pairs $(t_1,t_2)$ such that the imaginary part of $t_1/t_2$ is defined and positive. The group $SL_2(\Z)$ acts on $L \times \mathcal{C}(G)$ from the right by
\begin{equation}\label{sl2act}
((t_1,t_2),(g,h)) \cdot \left( \begin{array}{cc} a & b \\ c & d \end{array} \right)= ((at_1 + ct_2,bt_1+ dt_2), (g^dh^{-b},g^{-c}h^a)).
\end{equation}
And the group $G$ acts on $L \times \mathcal{C}(G)$ from the right by
\begin{equation}\label{cgact}
((t_1,t_2),(g,h))\cdot k := ((t_1,t_2),(k^{-1}gk,k^{-1}hk)), \quad \forall k\in G.
\end{equation}
Since the two group actions commute, we have a right action of the group \[G \times SL_2(\Z)\] on $L \times \mathcal{C}(G)$.

Let $\O(L)$ denote the ring of holomorphic functions on $L$. Let \[\O^j(L): =\{f\in \O(L) \mid f(\mu^2 t_1, \mu^2 t_2) = \mu^j f(t_1,t_2), \forall (t_1, t_2)\in L, \forall \mu\in \C^\ast\}, \]
where $\mu\in \mathbb{C}^\ast$ acts on the lattices by dilation and rotation. The $\C^\ast$-action on $L$ given by scaling both $t_1$ and $t_2$ induces a graded ring structure on $\O(L)$, i.e.
\[
\O(L) = \bigoplus_{j \in \Z} \: \O^j(L),
\]

\begin{definition}
The $SL_2(\Z)$-invariant elements of $\O^j(L)$ are called the \textit{weakly modular forms of weight $-j/2$}. We denote this subring of $\O^j(L)$ by $MF^j_{weak}$. \end{definition}

Let $G$ act on a space $X$ from the right, and denote by $X^{g,h} \subset X$ the subspace of points fixed by both $g$ and $h$. The action of $G$ on $X$ induces homemorphisms
\begin{equation}\label{act}
X^{g,h} \rightarrow X^{k^{-1}gk,k^{-1}hk}
\end{equation}
sending $x \mapsto x \cdot k$ for all $k \in G$.

As indicated in (\ref{cgact}), there is a $G$-action on $\mathcal{C}(G)$. We will use $\mathcal{C}[G]$ to denote the orbit space of the $G$-action  on $\mathcal{C}(G)$, and the symbol $[g,h]$ to denote the orbit of $(g,h)$. The stabilizer of  $(g,h)$ is the maximal subgroup \[C^{(2)}_G(g,h) \subset G\] that centralizes both $g$ and $h$. The $G$-action  induces a $C^{(2)}_G(g,h)$-action on $X^{g,h}$. 

\begin{proposition} For any $k \in G$, there is an isomorphism
\[
H^*_{C^{(2)}_G(k^{-1}gk,k^{-1}hk)}(X^{k^{-1}gk,k^{-1}hk}) \longrightarrow  H^*_{C^{(2)}_G(g,h)}(X^{g, h})
\]
induced by \eqref{act}. \end{proposition}

Then we are ready to give the definition of Devoto's equivariant elliptic cohomology.
\begin{definition}\label{dir}
For any integer $k$, the $k$-th Devoto's $G$-equivariant elliptic cohomology of a compact space $X$ is defined as the abelian group
\begin{equation}\label{devotoell}
\begin{array}{rcl}
 Ell^k_G(X) &:=&  \bigoplus\limits_{i+j = k} \left( \bigoplus\limits_{(g,h) \in \mathcal{C}(G)} H^i(X^{g, h})\otimes_\C \O^j(L) \right)^{G \times SL_2(\Z)} \\
&\cong& \bigoplus\limits_{i+j = k} \bigoplus\limits_{[g,h] \in \mathcal{C}[G]}  \left(H^i(X^{g, h})\otimes_\C \O^j(L) \right)^{C^{(2)}_G(g,h) \times SL_2(\Z)}
\end{array}
\end{equation}
where the isomorphism follows by choosing a representative pair $(g,h)$ for each conjugacy class $[g, h]$ in $\mathcal{C}(G)$.
\end{definition}

\begin{remark}
Note that $\{g,h\}$ and $\{g^dh^{-b},g^{-c}h^a\}$ generate the same subgroup of $G$.  And we have
\[
X^{g,h} = X^{g^dh^{-b},g^{-c}h^a} \quad \text{and} \quad C^{(2)}_G(g,h) = C^{(2)}_G(g^dh^{-b},g^{-c}h^a).
\] Thus, $SL_2(\Z)$ acts trivially on part of the cohomology $H^i(-)$ in \eqref{devotoell}.
\end{remark} 

\begin{remark}
If $G = \{e\}$ is the trivial group, then
\begin{equation}
Ell^k_{e}(X) =  \bigoplus_{i+j = k} H^i(X)\otimes_\C \O^j(L)^{SL_2(\Z)} = H^i(X) \otimes_\C MF^j_{weak}
\end{equation}
where the right hand side is the graded tensor product over $\C$ of the cohomology ring of $X$ with the graded ring of weakly modular forms of weight $-j/2$.
\end{remark}

\begin{remark}\label{pt}
If $X = \pt$, then
\[
Ell^k_G(\pt) =  \left( \bigoplus_{(g,h) \in \mathcal{C}(G)} \O^k(L) \right)^{G \times SL_2(\Z)} \cong \left( \bigoplus_{[g,h] \in \mathcal{C}[G]} \O^k(L) \right)^{SL_2(\Z)},
\]
which is a direct application of the isomorphism in Definition \ref{dir}. On the right hand side we obtain the direct sum of the ring of weakly modular forms of weight $-k/2$ with itself, indexed over all conjugacy classes of commuting pairs in $G$.
\end{remark}

\begin{remark}In \cite{ganter2007} Ganter discussed the equivariant elliptic cohomology $Ell^k$ by Devoto and provided a loop space model of equivariant Tate
K-theory motivated by Devoto's orbifold loop space. In \cite{huan2018}, the author constructed a loop space model motivated by Devoto's orbifold loop space with the circle rotation added. \end{remark}

\section{Twisted equivariant elliptic cohomology over $\mathbb{C}$}\label{twistedellold}

In \cite[Section 3]{Berwick2022}, Berwick-Evans constructed a twisted version of Devoto's equivariant elliptic cohomology with complex coefficients. We sketch his construction in this section. We construct in Section \ref{twistedChernS}  the Chern character from twisted quasi-elliptic cohomology to a variant of twisted Devoto's equivariant elliptic cohomology.

For convenience, the construction is expressed in term of normalised cocycles.
\begin{definition}
A 3-cocycle on $G$ with values in $U(1)$ is a map
\[
\alpha: G \times G \times G \rightarrow U(1)
\]
satisfying
\begin{equation}
\frac{\alpha(g_1,g_2,g_3)\alpha(g_0,g_1g_2,g_3)\alpha(g_0,g_1,g_2)}{\alpha(g_0g_1,g_2,g_3)\alpha(g_0,g_1,g_2g_3)} = 1
\end{equation}
for any $g_0,g_1,g_2,g_3 \in G$. Such a cocycle is called \textit{normalised} if it evaluates to $1$ on any triple containing the identity element $e \in G$. \end{definition}

Recall the value of $Ell^k_G$ of the single point space in Remark \ref{pt}. We can use $\alpha$ to twist the $G$-action on
\begin{equation}\label{pta}
\bigoplus_{(g,h) \in \mathcal{C}(G)} \O^k(L)
\end{equation}
by defining it to be
\begin{equation}\label{gro}
h\cdot_\alpha  f_{g_1,g_2} =  \frac{\alpha(g_2,h,g_1)\alpha(h,g_1,g_2)\alpha(g_1,g_2,h)}{\alpha(h,g_2,g_1)\alpha(g_1,h,g_2)\alpha(g_2,g_1,h)} f_{g_1,g_2}, \quad \forall h\in G,
\end{equation} where $f_{g_1, g_2}$ is the value of $f\in \O(L)$ at $(g_1, g_2)\in \mathcal{C}(G)$. 

\begin{remark}
In \cite[Section 1.4.3]{willerton2008}, Willerton gives the formulas of the transgression
\begin{equation}\label{Willertontransgression}
\tau_x\theta (g) := \frac{\theta(g,x)}{\theta(x,g)}, \qquad \tau_x \alpha(g,h) := \frac{\alpha(g,h,x) \alpha(x,g,h)}{\alpha(g,x,h)}, \quad \forall x \in G, \quad \forall g,h \in C_G(x)
\end{equation}
which sends a 2-cocycle $\theta \in Z^2(G)$ to a 1-cocycle $\tau_x \theta \in Z^1(C_G(x))$ and a 3-cocycle $\alpha \in Z^3(G)$ to a 2-cocycle $\tau_x \alpha \in Z^2(C_G(x))$. The relation between the $G$-action \eqref{gro} and the transgression \eqref{Willertontransgression} can be interpreted in this way:
\[
h\cdot_\alpha  f_{g_1,g_2} = \tau_{g_2}\alpha(h,g_1)  \tau_{g_2}\alpha(g_1,h)^{-1} f_{g_1,g_2} = \tau_{g_1} (\tau_{g_2} \alpha)(h) f_{g_1,g_2}.
\] 
This statement is not completely proven in \cite{willerton2008} and is completely proven in \cite[Section 3]{SatiSchreiber2022}.
\end{remark}

This is compatible with the $SL_2(\Z)$-action on $\mathcal{C}(G)$. 

We denote by the symbol 
\[
\bigoplus_{(g,h) \in \mathcal{C}(G)} \O^{k+\alpha}(L)
\] the group of holomorphic functions in \eqref{pta} equipped with the $G$-action twisted by $\alpha$. 
The $G$-invariants of this group are the collections of functions $(f_{g,h})_{(g,h) \in \mathcal{C}(G)}$ satisfying the following transformation property.
\[
h \cdot_\alpha f_{g_1,g_2} = f_{(h^{-1}g_1h,h^{-1}g_2h)} .
\]
\begin{definition}[\cite{Berwick2022}] 
For a manifold $X$ with an action of a finite group $G$, the $\alpha$-twisted version of Devoto's equivariant elliptic cohomology for a $G$-space $X$ is given by
\[
Ell_G^{k+\alpha}(X) := \bigoplus_{i+j = k} \left( \bigoplus_{ [g,h] \in \mathcal{C}[G]} H^i(X^{g,h}) \otimes \O^{j+\alpha}(L)  \right)^{C^{(2)}_G(g,h) \times SL_2(\Z)}.
\]
\end{definition}

\begin{remark}
In \cite[Section 4]{Berwick2022} Berwick-Evans constructs the induction  formula as well as other character formulas in the higher Hopkins-Kuhn-Ravenel character theory for the theory $Ell_G^{\ast+\alpha}(-)$. Moreover, he discussed its relation with physics.

 \end{remark}

\section{Twisted quasi-elliptic cohomology}\label{twistedqell}

\subsection{Definition}
In this section we define twisted quasi-elliptic cohomology $QEll^\alpha(X/\!\!/G)$ with $\alpha \in H^3(BG; U(1))$, which is constructed as the orbifold K-theory of a twisted orbifold $\Lambda^\alpha(X/\!\!/G)$ and is based on the construction in \cite[Section 3]{FreedQuinn1993}.

Let $G$ be a finite group and $X$ a $G-$space.
First, we show that each $\alpha \in H^3(BG;U(1))$ determines an element $\theta_g$ in $H^2(BC_G(g); U(1))$ for each conjugacy class $[g]$ in $G$. Let $e$ be the evaluation map
\[
e: B\Z \times \Map(B\Z,BG) \longrightarrow BG
\]
and let $\pi$ be the projection
\[
\pi: B\Z \times \Map(B\Z,BG) \twoheadrightarrow \Map(B\Z,BG).
\]
Define the class
\begin{equation}
\theta:= \pi_* e^* \alpha \in H^2(\Map(B\Z,BG);U(1)) \cong \bigoplus_{[g]} H^2(BC_G(g); U(1)), \label{thetadef}
\end{equation} 
where $[g]$ goes over all the conjugacy classes in $G$. The class $\theta$ is well-defined with degree two because the formalism of the Pontryagin-Thom construction means that the degree of $e^*\alpha$ drops by one when we push it forward along $\pi_*$. Note that we have also used the fact that the mapping space $\Map(B\Z,BG)$ is homotopy equivalent to
\[
\coprod_{[g]} BC_G(g).
\]
In this way, $\theta$ determines an element \[\theta_{g} \in H^2(BC_G(g);U(1))\] for each $[g]$.

We can now define the twisted orbifold. Each 2-cocycle $\theta_{g}$ determines a central extension
\[
1 \rightarrow \mathbb{T} \rightarrow C^\alpha_G(g)  \rightarrow C_G(g) \rightarrow 1
\]
with group multiplication given by
\[
(a,h)(b,k) = (a+b+\theta_g(h,k),hk),
\] for any $(a, h)$, $(b, k)$ in $C^\alpha_G(g)$.
We have a well-defined $C^\alpha_G(g)-$action  on $X^g$ \begin{equation}(a, h)\cdot x:= h\cdot x.\label{cextact}\end{equation}
For ease of notation, we have written $C^\alpha_G(g)$ in place of what is actually $C^{\theta_g}_G(g)$. Similar notational simplifications are also made in the rest part of the paper when there is no confusion. 

\begin{lemma}
Suppose that $\theta_g$ has order $n$, and let $l$ be the order of $g$. Then the order of $(0,g)$ 
divides $nl$.
\end{lemma}

\begin{proof}
Note that $g^l=e$. We have
\begin{align*}
nl(0,g) &= (\theta_g(g,g) + \theta_g(g,g^2) + ... + \theta_g(g,g^{nl-1}),g^{nl}) \\
&= ((n-1)(\theta_g(g,g) + \theta_g(g, g^2)+ ... + \theta_g(g,g^l))+ \theta_g(g,g) + ... + \theta_g(g,g^{l-1}), e) \\
&= (n(\theta_g(g,g) + \theta_g(g, g^2)+ ... + \theta_g(g,g^{l-1})) + (n-1)\theta_g(g,g^l),e) \\
&= (0,e).
\end{align*} 
The second equality holds because $g$ has order $l$, which means that $g^{ml + k} = g^k$ for all integers $m$ and $k$. The third equality holds since $(nl-1) = (n-1)l + l -1$, and we get the fourth equality just by rearranging terms. The final equality holds since $\theta_g$ has order $n$, and $\theta_g(g,e) = 0$ since $\alpha$ is normalised. Therefore, $nl(0,g)$ is equal to the identity element, and so the order of $(0,g)$ must divide $nl$.
\end{proof}
This result is cited in \cite[Remark 6.22]{dove2019}.


\begin{example}[Twisted Inertia Groupoid $I^{\alpha}(X/\!\!/G)$] T. Dove constructed the twisted Inertia groupoid in \cite{dove2019}.
The
twisted inertia groupoid $I^{tors}(X/\!\!/G)$ of the translation
groupoid $X/\!\!/G$ is the groupoid with

\textbf{objects}: the space $\coprod\limits_{g\in G}X^{g}$

\textbf{morphisms}: the space $\coprod\limits_{g\in
G}C_G^{\alpha}(g)\times X^g$.

For each $x\in X^g$, $C_G^{\alpha}(g)$  is the automorphism group of it. 

The twisted inertia groupoid is used in \cite[Section 6.4.1]{dove2019} to define the twisted equivariant Tate K-theory.\end{example}

\begin{example}[Twisted orbifold loop space]\label{twistedorbloop} 
In  \cite[Definition 2.3]{ganter2007} Ganter defined orbifold loop space \[\mathcal{L}(X/\!\!/G):=\coprod_{[g]}\mathcal{L}_gX/\!\!/C_G(g),\] via which equivariant Tate K-theory can be constructed. In this example we provide the twisted version of it. 

The space $\mathcal{L}_gX$ is the space $\mbox{Map}_{\mathbb{Z}/l}(\mathbb{R}/l\mathbb{Z},X) $ 
where $l$ is the order of $g$.
There is a well-defined
 $C_G^{\alpha}(g)-$action on $\mathcal{L}_gX$ by \[\gamma((a, h))(t)=\gamma(t+a)h\] for $\gamma \in \mathcal{L}_gX$ and $(a, h)\in C_G^{\alpha}(g)$. It's straightforward to
 check that $\gamma((a, h))$
 is indeed in $\mathcal{L}_gX$. The twisted orbifold loop space is defined as \[\mathcal{L}^{\alpha}(X/\!\!/G):=\coprod_{[g]}\mathcal{L}_gX/\!\!/C_G^{\alpha}(g).\]
 Note that on the space of constant loops $X^g$, the action by $C_G^{\alpha}(g)$ in (\ref{cextact}) covers that by $C_G(g)$.

 The twisted Inertia groupoid  $I^{\alpha}(X/\!\!/G)$ is the full subgroupoid of $\mathcal{L}^{\alpha}(X/\!\!/G)$ consisting of constant loops.
\end{example}

Let $\Lambda^{\alpha}_G(g)$ denote the quotient
\[
\R \times C^\alpha_G(g) / \langle(-1,(0,g))\rangle,
\] 
It
 fits into the short exact sequence \[1\longrightarrow C_G^{\alpha}(g) \longrightarrow
\Lambda^\alpha_G(g) \longrightarrow \mathbb{T}\longrightarrow 1.\]

We have the well-defined twisted orbifold
\begin{equation}
\Lambda^\alpha(X/\!\!/G) := \coprod_{g\in G^{tors}_{conj}} X^g/\!\!/\Lambda^\alpha_G(g).\label{lambdaalpha}
\end{equation}

In addition, we have the short exact sequence \begin{equation}1\longrightarrow \mathbb{T}\longrightarrow \Lambda^\alpha_G(g)\longrightarrow \Lambda_G(g)\longrightarrow 1. \label{rel:twist:lambda}\end{equation}
The surjective map in \eqref{rel:twist:lambda} gives the map between orbifolds
\begin{equation}
    \Lambda^\alpha(X/\!\!/G) \rightarrow \Lambda(X/\!\!/G) \label{lambdace}
\end{equation}
which sends a morphism $(x,[r,(a,h)])$ to $(x,[r,h])$. 

We would like to mention the map \eqref{lambdace} gives a $\mathbb{T}$-equivariant graded central extension in the sense of \cite{luecke2022}.

\begin{lemma}
The central extension $\Lambda^{\alpha}_{G}(g)$ is determined by the $2$-cocycle $\hat{\theta}_g \in Z^2(\Lambda_{G}(g))$ given by
\[
\hat{\theta}_g([\varsigma_2,t_2],[\varsigma_1,t_1]) = \theta_g([\varsigma_2],[\varsigma_1]).
\]
\end{lemma}

The following definition is \cite[Definition 7.1]{adem2003}.

\begin{definition} Let $\theta$ be an element in $H^2_G(X;\Z)$ and $G^{\theta}$ the group extension which represents it, 
$1\rightarrow \mathbb{T} \rightarrow G^{\theta} \rightarrow G\rightarrow 1$. Then $X$ is equipped with a $G^\theta$-action via the projection map $G^\theta \twoheadrightarrow G$. A \textit{$\theta$-twisted $G$-equivariant vector bundle over $X$} is defined to be a $G^{\theta}$-equivariant vector bundle $V$ over $X$ such that the central circle in $G^\theta$ acts by complex multiplication on the fibers of $V$. \label{twistedeqvectbd}
\end{definition}

Two $\theta$-twisted $G$-equivariant vector bundles over $X$ are isomorphic if and only if they are isomorphic as $G^\theta$-equivariant vector bundles. With this in mind, we state the following definition, which is \cite[Definition 7.2]{adem2003}.

\begin{definition} The $\theta$-twisted $G$-equivariant K-theory of a $G$-space $X$, denoted by $K^\theta_G(X)$, is defined to be the Grothendieck group of isomorphism classes of $\theta$-twisted $G$-equivariant vector bundles over $X$. \label{twistedeqK}
\end{definition}

Now we are ready to give the definition of twisted quasi-elliptic cohomology.

\begin{definition} \label{Def:twistedQEll} The twisted quasi-elliptic cohomology $QEll^{\alpha +\ast}_G(-)$ twisted by $\alpha\in H^3(BG; U(1))$
is defined by
\[QEll^{\alpha+\ast}_G(X):= \prod_{g\in G_{conj}}K^{\theta_g+\ast}_{\Lambda_G(g)}(X^g),\]
where each $\theta_g$ is the factor in \eqref{thetadef} corresponding to $g$. 
\end{definition}

\begin{remark} 
Note that $\theta$-twisted $G$-equivariant vector bundle is a special case of the twisted vector bundle in Definition 2.5 in \cite{gomi2017}
 with trivial $\mathbb{Z}/2$-grading. And the twisted equivariant K-theory in Definition \ref{twistedeqK} is a special case of Freed-Moore K-theory. Therefore, the Real twisted quasi-elliptic cohomology $QEllR^{\alpha+\ast}_G(X)$, which is constructed in \cite{huanyoung2022} as Freed-Moore K-theory of a Real version of the orbifold $\Lambda(X/\!\!/G)$, is a Real generalization of twisted quasi-elliptic cohomology. Twisted quasi-elliptic cohomology is equivalent to twisted Real quasi-elliptic cohomology of an orbifold with trivial $\Z/2$-grading.

 \label{Rmk:twisted_real}
\end{remark}

\begin{proposition}
We have the relation between $QEll^{\alpha+\ast}_G(-)$ and the twisted equivariant Tate K-theory $K^{\alpha + \ast}_{Tate}(-\git G)$  in \cite[Definition 6.21]{dove2019} as below.
\[QEll^{ \alpha + \ast }_G(X)\otimes_{\mathbb{Z}[q^{\pm}]}\mathbb{Z}((q))\cong K^{ \alpha + \ast}_{Tate}(X/\!\!/G)\]
\end{proposition}
\begin{proof}
This follows from the definition of both theories and Proposition \ref{tateqellequiv}. \end{proof}

\subsection{Examples and Properties}
In this section we provide some simple examples of twisted quasi-elliptic cohomology and some properties of it.

\begin{example}When $G$ is the trivial group and $g$ is the identity element, $QEll^{\ast}_G(X)=K^{\ast}_{\mathbb{T}}(X)$.
In this case, for any 3-cocycle $\alpha$, 
$QEll^{\alpha+\ast}_G(X)=K^{\theta_e + \ast}_{\mathbb{T}}(X)$.
\end{example}

\begin{example}Let $X$ be a CW complex with trivial $G$-action. In this case, for any 3-cocycle $\alpha$, by \cite[Lemma 7.3]{adem2003}, \begin{align*}QEll^{\alpha+\ast}(X)&=
\prod_{g\in G_{conj}}K^{\theta_g+\ast}_{\Lambda_G(g)}(X)\\
&\cong \prod_{g\in G_{conj}}K^{\ast}(X)\otimes R_{\theta_g}\Lambda_G(g)= K^{\ast}(X)\otimes \big(\prod_{g\in G_{conj}} R_{\theta_g}\Lambda_G(g)\big).\end{align*}
where $G_{conj}$ is the set of a family of representatives of the $G$-conjugacy classes in the finite group $G$ and $R_{\theta_g}\Lambda_G(g)$ denotes the ring of the $\theta_g-$representations of $\Lambda_G(g)$.

\end{example}

\begin{example}[Restriction map]Let $X$ be a $G-$space and $Y$ an $H-$space. Let $f: G\longrightarrow H$ denote a group homomorphism and let $h: X\longrightarrow Y$ denote a continuous map which is $G$-equivariant in the sense that \[h(g\cdot x)= f(g)\cdot h(x).\] 
From $f$ we can define the group homomorphisms \[f_g: \Lambda_G(g)\longrightarrow \Lambda_H(f(g)), \quad [a, t]\mapsto [f(a), t]\] for each $g\in G$.

Let $\alpha\in H^3(BH; U(1))$. Note that we have the commutative diagrams \begin{equation}
\xymatrix{ H^3(BH; U(1))\ar[r]^{f^*}\ar[d]^{\tau} &H^3(BG; U(1) )\ar[d]^{\tau}\\  H^2(\Map(B\Z,BH); U(1))\ar[r]^{f^*} &H^2(\Map(B\Z,BG); U(1)) \\ H^2(BC_H(f(g)); U(1))\ar[r]^{f^*}\ar@{^{(}->}[u] &H^2(BC_G(g); U(1)) \ar@{^{(}->}[u] }
  \end{equation}
Thus, \begin{equation}f^*(\tau(\alpha))=\tau(f^*(\alpha)); \quad f^*(\theta_{f(g)})=(f^*\theta)_g.  \end{equation} where $\tau(\alpha)=\prod\limits_{[h]}\theta_h$ with each $\theta_h\in H^2(BC_H(h); U(1)) $ and $\tau(f^*\alpha)=\prod\limits_{[g]}(f^*\theta)_g$ with each $(f^*\theta)_g\in H^2(BC_G(g); U(1) ) $.

Thus, we obtain a map for each $g\in G$ \[h_g^{\theta_g*}: K_{\Lambda_H(f(g))}^{\theta_{f(g)} + \ast }(Y^{f(g)})
\longrightarrow K_{\Lambda_G(g)}^{f^*(\theta_{f(g)}) + \ast }(X^g)\] and the restriction map
\begin{equation}h^*=\prod_{[h]} h_g^{\theta_g*}: QEll^{\alpha+\ast }_H(Y)
\longrightarrow QEll^{f^*{\alpha}+\ast }_G(X).\end{equation}

 \end{example}

In \cite[Section 3.4]{huanyoung2022}, the basic properties of twisted Real quasi-elliptic cohomology are discussed and proved in detail. By Remark \ref{Rmk:twisted_real}, we obtain the corresponding basic properties of twisted quasi-elliptic cohomology from them straightforwardly. We sketch the properties below.

\begin{proposition}
Let $G$ be a finite group and $H$ a subgroup of $G$. Let $X$ be a $G$-space and $\alpha\in H^3(BG; U(1))$. Then the change-of-group map 
\begin{equation}
\rho^{G}_{H}:
QEll^{\alpha+\ast}_{G}(X\times_{H} G)\longrightarrow QEll^{i^*\alpha+ \ast }_{H}(X\times_{H} G)\longrightarrow QEll^{ i^*\alpha + \ast}_{H}(X)
\end{equation}
is an isomorphism, where the first map is induced by the inclusion $i: H\hookrightarrow G$ of groups and the second map is induced by the inclusion \[ X\hookrightarrow X\times_{H} G, \quad x\mapsto [x, e].\]

\end{proposition}

\begin{example}
The induction for $QEll^{\alpha + \ast}$ can be defined as the composition
\begin{equation}
\mathcal{I}^{G}_{H}: QEll^{i^* \alpha + \ast }_{H}(X) \longrightarrow QEll^{\alpha + \ast}_{G}(X\times_{H} G)
\longrightarrow QEll^{\alpha + \ast}_{G}(X),
\end{equation}
where the second map is the induction map for $K$-theory induced by the finite covering $\Lambda((X\times_{H} G) /\!\!/ G) \rightarrow \Lambda( X /\!\!/ G )$ given on objects by $( [x, g], \sigma) \mapsto ( xg, \sigma)$ and morphisms by $([g^{\prime}, t], ([x, g], \sigma)) \mapsto ([g^{\prime}, t], (xg, \sigma))$. 
\end{example}

\subsection{Twisted Loop space}\label{twistedloopspace}

In this section we give a loop space construction of twisted quasi-elliptic cohomology other than the twisted orbifold loop space in Example \ref{twistedorbloop}. This is a twisted version of the loop space in Definition \ref{catequivloop}.

\begin{definition}For a $G$-space $X$, we define a  category of twisted equivariant loop space $Loop^{twist}(X/\!\!/ G)$. The objects of it are diagrams
\[\xymatrix{\mathbb{T} & P \ar[l]_{\rho} \ar@/^1pc/[rr]^{f'}  & Q\ar[l]^{\pi}\ar[r]_f &X }\] consisting of a principal $G$-bundle $P$ over $\mathbb{T}$, 
a principal $\mathbb{T}$-bundle $Q$ over $P$, a $G$-equivariant map $f'$ and $f=\pi\circ f'$.

A morphism from one object $\xymatrix{\mathbb{T} & P_1 \ar[l]_{\rho_1} \ar@/^1pc/[rr]^{f'_1}  & Q_1\ar[l]^{\pi_1}\ar[r]_{f_1} &X }$ to another object $\xymatrix{\mathbb{T} & P_2 \ar[l]_{\rho_2} \ar@/^1pc/[rr]^{f'_2}  & Q_2\ar[l]^{\pi_2}\ar[r]_{f_2} &X }$ is of the form $(\alpha, \beta, t)$
\[ \xymatrix{ Q_1 \ar[r]^{\alpha} \ar[d]_{\pi_1} & Q_2\ar[d]^{\pi_2} \\
P_1 \ar[r]^{\beta} \ar[d]_{\rho_1} &P_2\ar[d]^{\rho_2} \\ \mathbb{T}\ar[r]^{t} &\mathbb{T}} \]
where $t$ is a rotation on the circle $\mathbb{T}$, $\beta$ is a bundle isomorphism covering $t$ and $\alpha$ is a bundle isomorphism covering $\beta$. In addition, $f_1= f_2\circ \alpha$ and $f_1'= f_2'\circ \beta$.

\end{definition}

The constant twisted equivariant loops are those objects in $Loop^{twist}(X/\!\!/ G)$ \[\xymatrix{\mathbb{T} & P \ar[l]_{\rho} \ar@/^1pc/[rr]^{f'}  & Q\ar[l]^{\pi}\ar[r]_f &X }\] with both $f$ and $f'$ constant maps. If  $P$ is the principal $G$-bundle $P_g$ classified by  $g\in G$, then the image of $f$ 
consists of a single point $x\in X^g$. Each element $[ t, h]\in \Lambda_{G}(g)$ gives a bundle isomorphism from $P$ to itself covering the rotation $t$ of $\mathbb{T}$. In addition, each element $[a, [t, h]] \in \Lambda_{G}^{\theta_g}(g)$ gives a bundle isomorphism from $Q$ to itself covering $[t, h]$, 
where $\theta_g$ is some element in $ H^2(BC_G(g);U(1))$.   Thus, we have the conclusion below.

\begin{proposition}

The groupoid $\Lambda^{\alpha}(X/\!\!/G)$ in \eqref{lambdaalpha} is a subgroupoid of
\[Loop^{twist}(X/\!\!/G)\] with the constant loops $\prod\limits_{g\in G^{tors}_{conj}}X^g$ as objects.

\end{proposition}

\subsection{The Chern Character map for twisted quasi-elliptic cohomology}

In this subsection  we construct the Chern character maps for quasi-elliptic cohomology and twisted quasi-elliptic cohomology.

\subsubsection{The Chern Character map for $QEll$ }\label{QEllChernS}

We first define the Chern character of $QEll^{\ast}_{G}(X)$, based on the definition of Chern character map for any generalized cohomology theory given in \cite[Definition 4.9]{FSS2023}. 
The construction is given below.

Consider the diagram
\begin{equation}\xymatrix{1\ar[r] &C_G(g)\ar[r]\ar[d]_{=} & \mathbb{T}\times C_G(g) \ar[d]^{c_{g}}\ar[r] &\mathbb{T}\ar[r] \ar[d] &1\\ 1\ar[r] &C_G(g)
\ar[r] &\Lambda_G(g) \ar[r] &\mathbb{T}\ar[r] &1}\label{diag1chern} \end{equation} where the middle vertical map sends $(t, g)$ to $[lt, g]$ and the right vertical map sends $e^{2\pi i t}$
to $e^{2\pi i lt}$ with $l$ the order of $g$. Let \[c_{g}^*: K^*_{\Lambda_G(g)}(X^{g})\otimes \mathbb{C}\longrightarrow
K^*_{\mathbb{T}\times C_G(g)}(X^{g})\otimes \mathbb{C}\] denote the corresponding restriction map.

The Chern character map is constructed as the composition
\begin{align*}QEll^{\ast}_G(X)\otimes \mathbb{C}&=\prod_{[g]\in G_{conj}}K^\ast_{\Lambda_G(g)}(X^{g})\otimes \mathbb{C}\\
&\buildrel{c^*}\over\longrightarrow \prod_{[g]\in G_{conj}}K^\ast_{\mathbb{T}\times C_G(g)}(X^{g})\otimes \mathbb{C}\\
&\buildrel{\cong}\over\longrightarrow \prod_{[g]\in G_{conj}}K^\ast_{C_G(g)}(X^{g})\otimes \mathbb{Z}[q^{\pm}]\otimes \mathbb{C} \\
&\buildrel{AS}\over\longrightarrow \prod_{[g]\in G_{conj}}(\prod_{[h]\in C_G(g)_{conj}}(K^\ast(X^{g, h})\otimes \mathbb{C})^{C^{(2)}_G(g,h)}\otimes \mathbb{Z}[q^{\pm}])\\
&\buildrel{ch}\over\longrightarrow \prod_{[g, h] \in \mathcal{C}[G]}(H^{\ast}(X^{g, h})\otimes \mathbb{C})^{C^{(2)}_G(g, h)}\otimes \mathbb{Z}[q^{\pm}]\\
\end{align*}

The first map $c^*$ is the product $\prod\limits_{[g]\in G_{conj}}c^*_{g}$ of restriction maps. The property of equivariant K-theory implies that the second map is an isomorphism. The third map is the product of Atiyah-Segal maps in \cite[Theorem 2]{atiyah1989}
\[K_G^\ast(X)\otimes\mathbb{C}\buildrel\cong\over\longrightarrow \prod_{[g]\in G_{conj}} (K^\ast(X^{g})\otimes \mathbb{C})^{C_G(g)}.\] The fourth map is the product of Chern character maps of K-theory. 

Note that, other than the first map $c^*$, all the other maps in the composition are isomorphisms.

\subsubsection{The Chern Character map for twisted  $QEll$ }\label{twistedChernS} 

In this subsection we construct the twisted Chern character map of the twisted quasi-elliptic cohomology. By the definition of Chern character map for any generalized cohomology theory given in \cite[Definition 4.9]{FSS2023}, the twisted Chern character map deserves its name. 

Based on the map $c^*$ in the construction of Chern character map in Section \ref{QEllChernS}, we construct a map $p^*$. Let
\begin{equation}\label{twistchern1}
p_{\sigma}: \mathbb{T}\times C^{\theta_{\sigma}}_G(\sigma) \longrightarrow \Lambda^{\theta_{\sigma}}_G(\sigma)
\end{equation}
denote the map sending $(t, (a,g))$ to $[Nt, (a,g)]$, where $\theta_{\sigma}$ is the $2-$cocycle defined in Section \ref{twistedqell} and  $N$ is the order of $(0,\sigma)$ in $C^{\theta_{\sigma}}_G(\sigma)$. Let \[p_{\sigma}^*:K^{\theta_\sigma+ \ast}_{\Lambda_G(\sigma)}(X^{\sigma})\otimes \C \longrightarrow K^{\theta_\sigma + \ast }_{\mathbb{T}\times C_G(\sigma)}(X^{\sigma})\otimes \C\] denote the restriction map. Define \[p^*:=\prod_{[\sigma]\in G_{conj}} p_{\sigma}^*.\]

The twisted Chern character map is constructed as the composite
\begin{align*}\label{jer}
QEll^{\alpha+\ast}_G(X)\otimes \C&=\prod_{[\sigma]}K^{\theta_\sigma+\ast}_{\Lambda_G(\sigma)}(X^{\sigma})\otimes \C \\
&\buildrel{p^*}\over\longrightarrow \prod_{[\sigma]\in G_{conj}}K^{\theta_\sigma+\ast}_{\mathbb{T}\times C_G(\sigma)}(X^{\sigma})\otimes \C \\
&\buildrel{\cong}\over\longrightarrow \prod_{[\sigma]\in G_{conj}}K_{C_G(\sigma)}^{\theta_\sigma+\ast}(X^{\sigma})\otimes \C \otimes \mathbb{Z}[q^{\pm}] \\
&\buildrel{AS}\over\longrightarrow \prod_{[\sigma]\in G_{conj}}(\prod_{[\tau]\in C_G(\sigma)_{conj}}(K^\ast(X^{\sigma, \tau})\otimes L^{\theta_\sigma}_{\tau})^{C^{(2)}_{G}(\sigma, \tau)} \otimes\mathbb{Z}[q^{\pm}]) \\
&\buildrel{ch}\over\longrightarrow \prod_{[\sigma, \tau] \in \mathcal{C}[G]}(H^\ast(X^{\sigma, \tau}) \otimes L^{\theta_\sigma}_{\tau})^{C^{(2)}_G(\sigma,\tau)}  \otimes  \mathbb{Z}[q^{\pm}]
\end{align*}
where $h \in C^{(2)}_G(\sigma,\tau)$ acts on $L^{\theta_\sigma}_{\tau} \cong \C$ as multiplication by $\theta_\sigma (h,\tau)\theta_\sigma(\tau,h)^{-1}$.

In the composition above, the first map $p^*$ is the product of the restriction maps $p_{\sigma}^*$ induced by the group homomorphisms $p_{\sigma}$. 
If $M$ is the order of $(0,\sigma)$,     
then the kernel of $p^*_{\sigma}$
\[
\ker(p^*_\sigma) = \{(e^{2\pi i \frac{-m}{M}}, (\theta_\sigma(\sigma,\sigma) + ... + \theta_\sigma(\sigma,\sigma^{m-1}), \sigma^{m})) \in \mathbb{T}\times C^{\theta_{\sigma}}_G(\sigma) \mid m \in \Z \},
\]
which acts trivially on $X^\sigma$. The image of $p_{\sigma}^*$ is generated by the $\mathbb{T}\times C^{\theta_{\sigma}}_G(\sigma)$-vector bundles with trivial $\ker(p^*_\sigma)$-action on fibers. 
The second map is 
the isomorphism in \cite[Lemma 7.3]{adem2003} since the circle group $\T$ acts on each $X^\sigma$ trivially.
The third map is the twisted Atiyah-Segal map for twisted equivariant K-theory, which is proved in \cite[Theorem 7.4]{adem2003}.
The fourth map is the product of the Chern character maps of twisted K-theory. 

Below we provide the explicit description of the twisted Chern character map.
Let $\bigoplus\limits_{[\sigma]} E_\sigma$ be an element in twisted quasi-elliptic cohomology. Recall that $q$ is the character of the defining representation of $\mathbb{T}$. The pullback
\[
\bigoplus_{[\sigma]} p_\sigma^* E_\sigma
\]
splits as a direct sum
\begin{equation}\label{initial}
\bigoplus_{[\sigma]}  \bigoplus_{n\in\mathbb{Z}} (p^*E_\sigma)_n \otimes q^n,
\end{equation}
where we have written $p$ for $p_\sigma$, to simplify notation.

The twisted Atiyah-Segal map sends an element of the form \eqref{initial} to 
\begin{equation}\label{emm}
\bigoplus_{[\sigma,\tau]} \bigoplus_{n \in \Z} \bigoplus_{\xi(\tau)} \xi(\tau) (p^*E_\sigma)_n|_{X^{\sigma,\tau}} \otimes \theta_{\sigma}(-,\tau) \theta_\sigma(\tau,-)^{-1} \otimes q^n
\end{equation}
In \eqref{emm} above, $\xi(\tau)$ runs over the eigenvalues of $\tau$ and $n$ denotes the component where $\mathbb{T}$ acts by the eigenvalue $e^{ i \theta n}$.

Finally, the Chern character map sends the element \eqref{emm} to
\begin{equation}\label{d}
\left( \bigoplus_{[\sigma,\tau]} \bigoplus_n \bigoplus_{\xi(\tau)} \xi(\tau) ch( [(p^*E_\sigma)_n|_{X^{\sigma,\tau}}]) \right) \otimes \theta_{\sigma}(-,\tau) \theta_\sigma(\tau,-)^{-1} \otimes q^n.
\end{equation}

\begin{remark}
The composite of the final two maps of the twisted Chern character map is the same as the map in  \cite[Theorem 3.9]{FHT07}, tensored with $\Z[q^\pm]$. The twisted cohomology
\[
^{\theta_{\sigma}}H(X^{\sigma,\tau}; \, ^{\theta_{\sigma}}\mathcal{L}(\tau))
\]
which is the target of the map in \cite[Theorem 3.9]{FHT07} is the same as
\[
H(X^{\sigma, \tau}) \otimes L^{\theta_\sigma}_{\tau}.
\]
\end{remark}

\begin{remark}
    In \cite[Section 3.6]{huanyoung2022}, Young and the author construct the elliptic Pontryagin character for the Real twisted quasi-elliptic cohomology, which is the Real version of the twisted Chern character in this section.
\end{remark}

\section{Some Computations of Twisted Quasi-elliptic cohomology} \label{compute:twisted:qell}
Elliptic cohomology theory is usually very difficult to compute. The computation of twisted equivariant elliptic cohomology theory is even much harder. However, twisted quasi-elliptic cohomology theory, constructed as orbifold K-theory of a loop space, is computable. 

\subsection{Quasi-elliptic cohomology of $S^1$} \label{Ex:QEll:S1}

We start the computation with two basic examples.
\begin{example} [Rotation on $S^1$]

Let $\Z/N$ denote the cyclic group with $N$ elements. Consider the rotation action of it on the circle $S^1$. Since the rotation action is free.
The fix point space $ (S^1)^m$ by $m\in \Z/N$ is \[ (S^1)^m = \begin{cases} S^1, &\text{ if } m = 0 \\
\varnothing, &\text{  otherwise.}
\end{cases}
\]
In addition, 
\[
\Lambda_{\Z/N}(0) \cong   \Z/N \times \T.
\]

Thus, by the properties of equivariant K-theories \cite[Proposition 2.1,  Proposition 2.2]{segal1968b},   
\begin{align*}QEll_{\Z/N}(S^1) &\cong K_{\Lambda_{\Z/N}(0)} (S^1) \cong K_{\Z/N}(S^1)\otimes K_{\T}(\pt) \\
&\cong K(S^1/(\Z/N))\otimes \Z[q^{\pm}]
\cong K(S^1) \otimes \Z[q^{\pm}] = \Z \otimes \Z[q^{\pm}] \\
&= \Z[q^{\pm}] . 
\end{align*}

\end{example}

\begin{example}[Reflection on $S^1$]
Let $\Z/2 = \{1, \tau\}$ act on $S^1$ by reflection. The fixed point spaces are 
$$(S^1)^1= S^1; \quad (S^1)^{\tau} = S^0.$$
And
$$\Lambda_{\Z/2}(1) \cong \Z/2\times \T; \quad \Lambda_{\Z/2}(\tau) \cong (\Z/2\times \R) / \langle (\tau, -1)\rangle. $$
The two factors in $QEll_{\Z/2}(S^1)$ are computed below. By Example \ref{ex3.3:huan2018}, 
\begin{align*}K_{\Lambda_{\Z/2}(\tau)} ((S^1)^{\tau}) &\cong K_{\Lambda_{\Z/2}(\tau)}(\pt) \oplus K_{\Lambda_{\Z/2}(\tau)}(\pt) \\
&\cong \Z[q^{\pm}, x]/ \langle x^2-q\rangle \oplus \Z[q^{\pm}, x]/\langle x^2-q\rangle. \end{align*} 
In addition
\begin{equation}
K_{\Lambda_{\Z/2}(1)} ((S^1)^1) \cong K_{\Z/2} (S^1)\otimes K_{\T}(\pt) \cong \Z\otimes \Z[q^{\pm}] \cong  \Z[q^{\pm}].  \label{ex:ref:fac2}\end{equation}
We can view $S^1$ as two copies of $D^1$ glued via the boundary $S^0$. A complex vector bundles  over $S^1$ with the reflection is two copies of complex vector bundles over $D^1$ glued via the fibres on $S^0$  by identity. Thus,  the $\Z/2$-equivariant vector bundle over $S^1$ is equivalent to the trivial bundle over $S^1$. So we have $K_{\Z/2} (S^1) =\Z$ in \eqref{ex:ref:fac2}.

In conclusion, \begin{align*}
    QEll_{\Z/2}(S^1) & = K_{\Lambda_{\Z/2}(\tau)} ((S^1)^{\tau})  \oplus K_{\Lambda_{\Z/2}(1)} ((S^1)^1) \\
& \cong (\Z[q^{\pm}, x]/ \langle x^2-q\rangle \oplus \Z[q^{\pm}, x]/\langle x^2-q\rangle )\oplus \Z[q^{\pm}].
\end{align*}

\end{example}

\subsection{Quasi-elliptic cohomology of $4$-sphere acted on by a finite subgroup of $SU(2)$} \label{ex:qell:s4}

Quasi-elliptic cohomology theory is conjectured in \cite{SatiSchreiber2022}
as a particularly suitable approximation to equivariant
$4$-th $Cohomotopy$, which classifies the
charges carried by M-branes in M-theory in a way that is analogous to
the traditional idea that complex K-theory classifies the charges of D-branes in
string theory. 
In this subsection we compute quasi-elliptic cohomology of 4-spheres under the action by some finite subgroups that are the most interesting isotropy groups where the M5-branes may sit.

Let $G$ denote any finite subgroup of $SU(2)$. 
First we explain how the group $G$ acts on $S^4$.
We have the standard orthogonal $SO(5)$-action on $\R^5$ and also on the subspace $S^4\subset \R^5$. The covering map \[Spin(5) \longrightarrow SO(5)\] makes $S^4$ a well-defined $Spin(5)$-space. 
The $G$-action on $S^4$ is induced by  the composition 
\begin{equation} i_G: G\hookrightarrow Spin(3)\buildrel{p_1}\over\longrightarrow Spin(3)\times Spin(3) =Spin(4) \hookrightarrow Spin(5) \label{gpact:def}\end{equation}  where $p_1$ is the projection to the first factor of the product group. 

We give the explicit formula of the $G$-action below.
The group of unit quaternions is isomorphic to $SU(2)\cong Spin(3)$ via the correspondence \[a+bi+cj+dk \mapsto \left[ {\begin{array}{cc} a+bi & c+di \\  -c+di & a-bi
\end{array}} \right].  \] In view of this, 
$Spin(4)$ can be described as the group \[ \{ \left[ {\begin{array}{cc} q & 0 \\  0 & r
\end{array}} \right] \mid q, r\in \mathbb{H}, |q|=|r|=1.\}, \] and $Spin(5)$ can be identified with the  quaternionic unitary group.
Thus, as indicated in \cite[pp.263]{porteous_1995}, the inclusion  $Spin(4)\hookrightarrow Spin(5)$ is given by the formula
\begin{equation} \left[ {\begin{array}{cc} q & 0 \\  0 & r
\end{array}} \right] \mapsto \left[ {\begin{array}{cc} q & 0 \\  0 & r
\end{array}} \right]. \label{spin4act} \end{equation}

In addition, as shown in \cite[pp.151]{porteous_1995}, the rotation of $\mathbb{R}^4$ represented by 
\[\left[ {\begin{array}{cc} q & 0 \\  0 & r
\end{array}} \right] \in Spin(4)\] is given by the map \begin{equation} \label{spin(4)action} \left[ {\begin{array}{cc} y & 0 \\  0 & \overline{y}
\end{array}} \right] \mapsto 
\left[ {\begin{array}{cc} q & 0 \\  0 & r
\end{array}} \right] \left[ {\begin{array}{cc} y & 0 \\  0 & \overline{y}
\end{array}} \right] \widehat{\left[ {\begin{array}{cc} q & 0 \\  0 &r
\end{array}} \right]}^{-1} = \left[ {\begin{array}{cc} qy\overline{r} & 0 \\  0 & r\overline{y}\overline{q}
\end{array}} \right].\end{equation} where $\R^4$ is identified with the linear space \[\{ \left[ {\begin{array}{cc} y & 0 \\  0 & \overline{y}
\end{array}} \right] \mid  y\in \mathbb{H}.\} .\]

Then, the group $Spin(4)\subset Spin(5)$ acts on $S^4\subset \mathbb{R}^5$ via the composition \begin{equation} Spin(4)\rightarrow SO(4)\xrightarrow{A\mapsto \left[ {\begin{array}{cc} A & 0 \\  0 & 1
\end{array}} \right]}SO(5) \label{spins5} \end{equation}
with the standard orthogonal action.

\bigskip

Before we compute the examples, we recall the  classification of the finite subgroups of $Spin(3)\cong SU(2)$. There are many references for the classification, \cite[Chapter XIII]{dickson2014algebraic}, \cite{Stekolshchik_Coxeter_Mckay}, \cite{nlab:finite_rotation_group} etc. The finite subgroups of $SU(2)$ are classified as: \begin{itemize}\item 
the cyclic group of order $n$\[G_n:=\{\left[ {\begin{array}{cc} e^{\frac{2\pi ki}{n}} & 0 \\  0 &e^{-\frac{2\pi ki}{n}} 
\end{array}} \right] \mid k\in\Z \};\]
\item  the dicyclic group of order $4n$ \[2D_{2n}:=\langle G_{2n}, \left[ {\begin{array}{cc} 0 & 1 \\  -1 & 0
\end{array}} \right] \rangle; \]  
\item  the binary tetrahedral group $E_6$;
\item the binary octahedral group $E_7$; 
\item the binary icosahedral group $E_8$; \end{itemize}
where $n$
is any positive integer. 

In the computation below, we use the symbol \[A_\theta\] to denote the matrix \[
\left[ {\begin{array}{cc} e^{\theta i} & 0 \\  0 &e^{-\theta i} 
\end{array}} \right] .\]
\begin{example} \label{GnactS4}

In this example, we compute $QEll_{G_n}(S^4)$, where \[G_n:=\{\left[ {\begin{array}{cc} e^{\frac{2\pi ki}{n}} & 0 \\  0 &e^{-\frac{2\pi ki}{n}} 
\end{array}} \right] \mid k\in\Z \}.\] 
An element in $G_n$ is of the form $A_{\theta}$ with $\theta =  \frac{2\pi m}{n}$ for some integer $m$.
Since $G_n$ is abelian, each conjugacy class in it has exactly one element.
Below we compute the factor in $QEll_{G_n}(S^4)$ corresponding to each conjugacy class.

If $n\nmid m$, $(S^4)^{A_{\theta}} \cong S^0$. 
We have  \begin{align*} K_{\Lambda_{G_n}(A_\theta)}((S^4)^{A_{\theta}} ) &\cong K_{\Lambda_{\mathbb{Z}/n} (m)}(S^0) \cong R\Lambda_{\mathbb{Z}/n}(m) \oplus 
R\Lambda_{\mathbb{Z}/n}(m) \\ 
&\cong \mathbb{Z}[q^{\pm}, x]/\langle x^n-q^m \rangle \oplus \mathbb{Z}[q^{\pm}, x]/\langle x^n-q^m \rangle. \end{align*}
The last step is by the computation in Example \ref{ex3.3:huan2018}. 

If $n \mid m$, $(S^4)^{A_{\theta}} \cong S^4$. \begin{align*}
    K_{\Lambda_{G_n}(A_{\theta})}((S^4)^{A_{\theta}}) &\cong K_{\mathbb{Z}/n\times \mathbb{T}}(S^4) \cong 
    K_{\mathbb{Z}/n}(S^4)\otimes \mathbb{Z}[q^{\pm}] \\ &\buildrel{(*)}\over\cong K_{\mathbb{Z}/n}(S^0)\otimes \mathbb{Z}[q^{\pm}] \cong
    (R(\mathbb{Z}/n)\oplus R(\mathbb{Z}/n))\otimes \mathbb{Z}[q^{\pm}] \\ &\cong \mathbb{Z}[q^{\pm}, x]/\langle x^n-1\rangle \oplus\mathbb{Z}[q^{\pm}, x]/\langle x^n-1\rangle .
\end{align*} where the isomorphism $(*)$ is obtained from the equivariant Bott periodicity \cite[Theorem 4.3]{atiyah1968_Bott} and that the action of $\mathbb{Z}/n$ on the north pole and south pole is trivial.

In conclusion, \begin{align*}
    QEll_{G_n}(S^4) &= \prod^n_{m=0}  K_{\Lambda_{G_n}(A_{ \frac{2\pi  m}{n}})}((S^4)^{A_{  \frac{2\pi  m}{n}}} ) \\
    &\cong \prod^n_{m=0} \mathbb{Z}[q^{\pm}, x]/\langle x^n-q^m \rangle \oplus \mathbb{Z}[q^{\pm}, x]/\langle x^n-q^m \rangle.
\end{align*}
\end{example}

\begin{example} \label{2D2nS4}

In this example we compute $QEll_{2D_{2n}} (S^4)$, where $$2D_{2n}=\langle G_{2n}, \left[ {\begin{array}{cc} 0 & -1 \\  1 &  0 
\end{array}} \right]\rangle $$ with $n$ a positive integer, and $G_{2n}$ is the cyclic group generated by \[ A_{\frac{2\pi}{2n}}. \] 
We will use $\tau$ to denote the matrix \[\left[ {\begin{array}{cc} 0 & -1 \\  1 &  0 
\end{array}} \right].\] 

In $2D_{2n}$  there are $n+3$ conjugacy classes. They are:
\begin{enumerate}
\item $\{I\}$, \item $\{-I\}$, \item 
$\{A_{\frac{\pi}{n}}, A^{-1}_{\frac{\pi}{n}}\}$, 
$\{A^2_{\frac{\pi}{n}}, A^{-2}_{\frac{\pi}{n}}\}$, $\cdots$, $\{A^{n-1}_{\frac{\pi}{n}}, A^{-(n-1)}_{\frac{\pi }{n}}\}$,  
\item 
$\{\tau, \tau A^2_{\frac{\pi}{n}}, \tau A^4_{\frac{\pi}{n}}\cdots \tau A^{2n-2}_{\frac{\pi}{n}}\} $, 
\item $\{\tau A_{\frac{\pi}{n}}, \tau A^3_{\frac{\pi}{n}}, \cdots \tau A^{2n-1}_{\frac{\pi}{n}} \}$, \end{enumerate} where the first two form the centre of the group. 

Next we compute the factor in $QEll_{2D_{2n}} (S^4)$ corresponding to each conjugacy class below.
\begin{enumerate}
    \item 
First we consider the conjugacy class represented by $I$.
The centraliser $$C_{2D_{2n}}(I) = 2D_{2n},$$ and, thus,  
$\Lambda_{2D_{2n}}(I)\cong 2D_{2n} \times \mathbb{T}$. The corresponding equivariant K-theory
\begin{align*}
K_{\Lambda_{2D_{2n}}(I)}((S^4)^I) &\cong K_{2D_{2n} \times \mathbb{T}}(S^4) \cong K_{2D_{2n}}(S^4) \otimes \mathbb{Z}[q^{\pm}] \\
& \buildrel{(\ast)}\over\cong K_{2D_{2n}} (S^0) \otimes \mathbb{Z}[q^{\pm}] \cong (R(2D_{2n})\oplus R (2D_{2n})) \otimes \mathbb{Z}[q^{\pm}],
\end{align*} where the step $(\ast)$ is also obtained from the equivariant Bott periodicity.

\item
Then we consider the conjugacy class represented by $-I$.
The centraliser $$C_{2D_{2n}}(-I) = 2D_{2n}.$$ 
And the fixed point space $(S^4)^{-I}=S^4$. 
The group $2D_{2n}$ fits into the short exact sequence\[0\longrightarrow \mathbb{Z}/2 \longrightarrow 2D_{2n} \buildrel{\pi}\over\longrightarrow D_{2n}\longrightarrow 0, \] where $D_{2n}$ is the dihedral group with $2n$ elements.
Then, by Lemma \ref{dcld2}, 
\begin{align*}
K_{\Lambda_{2D_{2n}}(-I)}((S^4)^{-I}) &\cong K_{\Lambda_{2D_{2n}}(-I)}(S^0) \cong  K_{\Lambda_{2D_{2n}}(-I)}(\pt)\oplus K_{\Lambda_{2D_{2n}}(-I)}(\pt)\\
&\cong K_{\Lambda_{D_{2n}}(I)}(\pt ) \oplus K_{\Lambda_{D_{2n}}(I)}(\pt ) 
\oplus K^{[\widetilde{\Lambda_{D_{2n}}(I)}_{\rho} ] +*}_{\Lambda_{D_{2n}}(I)}(\pt ) 
\oplus K^{[\widetilde{\Lambda_{D_{2n}}(I)}_{\rho} ] +*} K_{\Lambda_{D_{2n}}(I)}(\pt ) 
\\ &\cong K_{D_{2n}\times \T}(\pt ) \oplus K_{D_{2n}\times \T}(\pt ) 
\oplus K^{[\widetilde{(D_{2n}\times \T)}_{\rho} ] +*}_{D_{2n}\times \T}(\pt ) 
\oplus K^{[\widetilde{(D_{2n}\times \T)}_{\rho} ] +*}_{D_{2n}\times \T}(\pt ) \\
& \cong (R(D_{2n})\oplus R(D_{2n}))\otimes \mathbb{Z}[q^{\pm}] \oplus (R_{[\widetilde{D_{2n}}_{\rho}]}(D_{2n})\oplus R_{[\widetilde{D_{2n}}_{\rho}]}(D_{2n}))\otimes \Z[q^{\pm}].
\end{align*} where $\rho$ is the sign representation of $\Z/2$.

\item 
Then  we consider the conjugacy class represented by the element $A^m_{\frac{\pi }{n}}$, $m= 1, 2,\cdots n-1$.
The centralizer \[C_{2D_{2n}}(A^m_{\frac{\pi }{n}}) = G_{2n} \cong \mathbb{Z}/{2n}.\]

Thus, $\Lambda_{2D_{2n}}(A^m_{\frac{\pi }{n}})\cong \Lambda_{\mathbb{Z}/2n}(2m)$.
In addition, the fixed point space
\[(S^4)^{A^m_{\frac{\pi }{n}}} \cong S^0.\] 

So we have \begin{align*} 
K_{\Lambda_{2D_{2n}}(A^m_{\frac{\pi }{n}})} ((S^4)^{A^m_{\frac{\pi }{n}}}) & \cong  K_{\Lambda_{\mathbb{Z}/2n}(2m)} (S^0) \cong 
R\Lambda_{\mathbb{Z}/2n}(2m) \oplus R \Lambda_{\mathbb{Z}/2n}(2m) \\
&\cong \mathbb{Z}[x, q^{\pm}]/\langle x^{2n}- q^{2m}\rangle \oplus \mathbb{Z}[x, q^{\pm}]/\langle x^{2n}- q^{2m}\rangle.
\end{align*}

\item 

Then we consider the conjugacy class represented by $\tau$.  
The centralizer $C_{2D_{2n}}(\tau)= \langle \tau \rangle \cong \mathbb{Z}/4$ and the fixed point space $(S^4)^{\tau}$ is $S^0$.

Thus, \begin{align*}
    K_{\Lambda_{2D_{2n}}(\tau)}((S^4)^{\tau}) &\cong K_{\Lambda_{\mathbb{Z}/4}(1)} (S^0) \cong R\Lambda_{\mathbb{Z}/4}(1) \oplus R\Lambda_{\mathbb{Z}/4}(1) \\
    &\cong \Z[x, q^{\pm}]/\langle x^4-q\rangle \oplus \Z[x, q^{\pm}]/ \langle x^4-q\rangle .
\end{align*}

\item 

Then we deal with the last conjugacy class, which is represented by $\tau A_{\frac{2\pi}{2n}}$. There are $n$ elements in the conjugacy class, thus, the centralizer of $\tau A_{\frac{2\pi}{2n}}$ has 4 elements. Then it's direct to check that \[ C_{2D_{2n}}(\tau A_{\frac{2\pi}{2n}}) = \langle \tau A_{\frac{2\pi}{2n}}\rangle \cong \Z/4.\] In addition, the fixed point space $(S^4)^{\tau A_{\frac{2\pi}{2n}}} $ is $S^0$.

Thus, \begin{align*}
    K_{\Lambda_{2D_{2n}}(\tau  A_{\frac{2\pi}{2n}})}((S^4)^{\tau  A_{\frac{2\pi}{2n}}}) &\cong K_{\Lambda_{\mathbb{Z}/4}(1)} (S^0) \cong  R\Lambda_{\mathbb{Z}/4}(1)\oplus R\Lambda_{\mathbb{Z}/4}(1)  \\
    &\cong \Z[x, q^{\pm}]/\langle x^4-q\rangle \oplus \Z[x, q^{\pm}]/\langle x^4-q\rangle.
\end{align*}

\end{enumerate}

Thus, in conclusion, \begin{align*}
    QEll_{2D_{2n}} (S^4) = &K_{\Lambda_{2D_{2n}}(I)}((S^4)^I) \times K_{\Lambda_{2D_{2n}}(-I)}((S^4)^{-I})  \\
    &\times \prod_{m=1}^{n-1}K_{\Lambda_{2D_{2n}}(A^m_{\frac{\pi }{n}})} ((S^4)^{A^m_{\frac{\pi }{n}}}) \\
    &\times K_{\Lambda_{2D_{2n}}(\tau)}((S^4)^{\tau}) 
    \times  K_{\Lambda_{2D_{2n}}(\tau  A_{\frac{2\pi}{2n}})}((S^4)^{\tau  A_{\frac{2\pi}{2n}}}) \\
    \cong  &(R(2D_{2n})\oplus R (2D_{2n})) \otimes \mathbb{Z}[q^{\pm}] \\ & \times (R(D_{2n})\oplus R(D_{2n}))\otimes \mathbb{Z}[q^{\pm}] \oplus (R_{[\widetilde{D_{2n}}_{\rho}]}(D_{2n})\oplus R_{[\widetilde{D_{2n}}_{\rho}]}(D_{2n}))\otimes \Z[q^{\pm}] \\ &\times \prod_{m=1}^{n-1}  \mathbb{Z}[x, q^{\pm}]/\langle x^{2n}- q^{2m}\rangle \oplus \mathbb{Z}[x, q^{\pm}]/\langle x^{2n}- q^{2m}\rangle \\ &\times  \Z[x, q^{\pm}]/\langle x^4-q\rangle \oplus \Z[x, q^{\pm}]/ \langle x^4-q\rangle \\
    & \times \Z[x, q^{\pm}]/\langle x^4-q\rangle \oplus \Z[x, q^{\pm}]/\langle x^4-q\rangle, 
\end{align*}  where $\rho$ is the sign representation of $\Z/2$.

\end{example}

\begin{example}\label{E6S4}
In this example we compute $QEll_{E_6}(S^4)$ where $E_6$
is the binary tetrahedral group. 
The quaternion representation of $E_6$ is given explicitly at \cite{Phillips_Tetrahedral} and  \cite{QR:BiTetraGrp}.

We can compute the conjugacy classes in $E_6$ explicitly. A list of representatives are given in Figure \ref{E6:conj}. This list can be obtained by direct computation. A multiplication table for the binary tetrahedral group is given here \cite{multi:BiTetraGrp}. For the convenience of the readers, we apply the same symbols of the elements as those in \cite{multi:BiTetraGrp} and \cite{QR:BiTetraGrp}. 
\begin{figure} \begin{center} 
\begin{tabular}{|c | c | c |} 
 \hline 
 A representative of the conjugacy class   & Conjugacy class & Order 
 \\ \hline
$1$ & $ \{ 1 \}$ & $1$ \\
$-1$ & $\{-1\}$ & $2$ \\
$i$ & $\{\pm i, \pm j,  \pm k\} $ & $4$ \\
$a$ & $\{a, b, c, d\} $ & $6$ \\
$-a$ & $\{-a, -b, -c, -d\}$ & $3$ \\
$a^2$ & $\{a^2, b^2, c^2, d^2\}$ & $3$ \\
$-a^2$ & $ \{-a^2, -b^2, -c^2, -d^2\} $  & $6$ \\
 \hline
\end{tabular} \caption{Conjugacy classes of $E_6$}\label{E6:conj}
\end{center} \end{figure}


By the multiplication table \cite{multi:BiTetraGrp} and direct computation, we obtain the centralizers of each representative and the corresponding fixed point space, as in Figure \ref{E_6: CF}.
\begin{figure}
\begin{center}
\begin{tabular}{|c | c | c | } 
 \hline Representatives $\alpha$  &Centralizers &Fixed point spaces\\
 of Conjugacy classes & $C_{E_6}(\alpha)$ & $(S^4)^{\alpha}$
 \\ \hline
$1$ & $ E_6$ & $S^4$\\
$-1$ & $E_6$ & $ S^0$\\
$i$ & $\{\pm 1, \pm i\} \cong \mathbb{Z}/4$ & $S^0$\\
$a$ & $\{\pm 1, \pm a, \pm a^2\} \cong \mathbb{Z}/6$ & $S^0$ \\
$-a$ & $\{\pm 1, \pm a, \pm a^2\} \cong \mathbb{Z}/6$ & $S^0$\\
$a^2$ & $\{\pm I, \pm a, \pm a^2\} \cong \mathbb{Z}/6$  & $S^0$ \\
$-a^2$ & $\{\pm 1, \pm a, \pm a^2\} \cong \mathbb{Z}/6$ & $S^0$ \\
 \hline
\end{tabular} \caption{Centralizers and fixed point spaces} \label{E_6: CF}
\end{center} \end{figure}

Then the factors in $QEll_{E_6}(S^4)$ corresponding to each conjugacy class is computed below.
\begin{enumerate}
\item For the conjugacy class represented by $1$,
\begin{align*}
K_{\Lambda_{E_6}(1)} ((S^4)^1) &\cong K_{E_6\times \mathbb{T}}(S^4) \cong K_{E_6}(S^4)\otimes \mathbb{Z}[q^{\pm}] \\
&\cong K_{E_6}(S^0) \otimes \mathbb{Z}[q^{\pm}] \cong (R(E_6) \oplus R(E_6)) \otimes \mathbb{Z}[q^{\pm}].
\end{align*}

\item Then we compute the factor corresponding to the conjugacy class represented by $-1$.
We have the commutative diagram
\begin{equation}
\xymatrix{ 0\ar[r] &\mathbb{Z}/2 \ar[r] \ar@{=}[d] &E_6 \ar[r]^{\pi} \ar@{^{(}->}[d] &T_6\ar[r] \ar@{^{(}->}[d] &0 \\ 
0\ar[r] &\mathbb{Z}/2 \ar[r] &Spin(3) \ar[r]^{\pi} &SO(3)\ar[r] &0 },
\end{equation} where $T_6$ is the tetrahedral group and both the horizontal sequences are exact.

Then,   by Lemma \ref{dcld2}, we have
\begin{align*}
K_{\Lambda_{E_6}(-1)} ((S^4)^{-1}) &\cong K_{\Lambda_{T_6}(1)}(S^0) \oplus K^{ [\widetilde{\Lambda_{T_6}(1)}_{\rho}]+*}_{\Lambda_{T_6}(1)}(S^0) \\
&\cong K_{T_6\times \mathbb{T}}(S^0) \oplus K^{ [\widetilde{(T_6\times \mathbb{T})}_{\rho}]+*}_{T_6\times \mathbb{T}}(S^0) \\
&\cong K_{T_6}(S^0)\otimes \Z[q^{\pm}] \oplus K^{ [\widetilde{(T_6)}_{\rho}]+*}_{T_6}(S^0)\otimes \Z[q^{\pm}] \\
&\cong (R(T_6)\oplus R (T_6) \oplus R_{[\widetilde{(T_6)}_{\rho}]} (T_6) \oplus R_{[\widetilde{(T_6)}_{\rho}]} (T_6)) \otimes \Z[q^{\pm}].
\end{align*} where $\rho$ is the sign representation of $\Z/2$.

\item For the conjugacy class represented by $i$, we have 
\begin{align*}
K_{\Lambda_{E_6}(i)}((S^4)^i) &\cong K_{\Lambda_{\mathbb{Z}/4}(1)}(S^0) \cong R(\Lambda_{\mathbb{Z}/4}(1)) \oplus R(\Lambda_{\mathbb{Z}/4}(1)) 
\\ &\cong \Z[x, q^{\pm}]/\langle x^4- q\rangle \oplus \Z[x, q^{\pm}]/\langle x^4- q\rangle.
\end{align*}

\item For the conjugacy class represented by $a$, we have 
\begin{align*}
K_{\Lambda_{E_6}(a)}((S^4)^a) &\cong K_{\Lambda_{\mathbb{Z}/6}(1)}(S^0) \cong R(\Lambda_{\mathbb{Z}/6}(1)) \oplus R(\Lambda_{\mathbb{Z}/6}(1)) 
\\ &\cong \Z[x, q^{\pm}]/\langle x^6- q\rangle \oplus \Z[x, q^{\pm}]/\langle x^6- q\rangle.
\end{align*}

\item For the conjugacy class represented by $-a$, we have \begin{align*}
K_{\Lambda_{E_6}(-a)} ((S^4)^{-a}) &\cong K_{\Lambda_{\Z/6}(4)} (S^0) \cong R(\Lambda_{\Z/6}(4))\oplus R (\Lambda_{\Z/6}(4)) \\
&\cong \Z[x, q^{\pm}] /\langle x^6- q^4\rangle \oplus \Z[x, q^{\pm}] /\langle x^6- q^4\rangle.    
\end{align*}

\item For the conjugacy class represented by $a^2$, we have \begin{align*}
K_{\Lambda_{E_6}(a^2)} ((S^4)^{a^2}) &\cong K_{\Lambda_{\Z/6}(2)} (S^0) \cong R(\Lambda_{\Z/6}(2))\oplus R (\Lambda_{\Z/6}(2)) \\
&\cong \Z[x, q^{\pm}] /\langle x^6- q^2\rangle \oplus \Z[x, q^{\pm}] /\langle x^6- q^2\rangle    
\end{align*}

\item For the conjugacy class represented by $-a^2$, we have 
\begin{align*}
K_{\Lambda_{E_6}(-a^2)}((S^4)^{-a^2}) &\cong K_{\Lambda_{\mathbb{Z}/6}(5)}(S^0) \cong  R(\Lambda_{\mathbb{Z}/6}(5)) \oplus R(\Lambda_{\mathbb{Z}/6}(5))
\\ &\cong \Z[x, q^{\pm}]/\langle x^6- q^5\rangle \oplus \Z[x, q^{\pm}]/\langle x^6- q^5\rangle .  
\end{align*}

\end{enumerate}

Thus, in conclusion, \begin{align*}
    QEll_{E_6}(S^4) = &K_{\Lambda_{E_6}(1)} ((S^4)^1) \times K_{\Lambda_{E_6}(-1)} ((S^4)^{-1}) \times K_{\Lambda_{E_6}(i)}((S^4)^i) \\
    & \times K_{\Lambda_{E_6}(a)}((S^4)^a) \times K_{\Lambda_{E_6}(-a)} ((S^4)^{-a}) 
       \times K_{\Lambda_{E_6}(a^2)} ((S^4)^{a^2}) \\ &\times K_{\Lambda_{E_6}(-a^2)}((S^4)^{-a^2})
\\
\cong 
   &(R(E_6) \oplus R(E_6)) \otimes \mathbb{Z}[q^{\pm}]\\
  &\times (R(T_6)\oplus R (T_6) \oplus R_{[\widetilde{(T_6)}_{\rho}]} (T_6) \oplus R_{[\widetilde{(T_6)}_{\rho}]} (T_6)) \otimes \Z[q^{\pm}]
   \\ & \times  \Z[x, q^{\pm}]/\langle x^4- q\rangle \oplus \Z[x, q^{\pm}]/\langle x^4- q\rangle \\
   & \times \Z[x, q^{\pm}]/\langle x^6- q\rangle \oplus \Z[x, q^{\pm}]/\langle x^6- q\rangle \\
    &\times  \Z[x, q^{\pm}] /\langle x^6- q^4\rangle \oplus \Z[x, q^{\pm}] /\langle x^6- q^4\rangle \\
    &\times \Z[x, q^{\pm}] /\langle x^6- q^2\rangle \oplus \Z[x, q^{\pm}] /\langle x^6- q^2\rangle\\
&    \times \Z[x, q^{\pm}]/\langle x^6- q^5\rangle \oplus \Z[x, q^{\pm}]/\langle x^6- q^5\rangle, 
\end{align*}  where $\rho$ is the sign representation of $\Z/2$.
    
\end{example}

\begin{example} \label{E7S4}
In this example we compute $QEll_{E_7}(S^4)$ where $E_7$ is the 
binary octahedral group. 

A presentation of $E_7$ is given as  
\[E_7= \langle s, t \mid r^2=s^3= t^4 = rst = -1\rangle.  \]
We can get immediately that $r=st$. Equivalently, there is a quaternion presentation of $E_7$ given by the embedding  \[E_7 \rightarrow  \mathbb{H}\] 
 sending $s$ to
$\frac{1}{2}(1+i+j+k)$,  $t$ to $\frac{1}{\sqrt{2}}(1+i)$, and $r$ to $\frac{1}{\sqrt{2}} (i+j)$. 

By \cite{McKay1980} and direct computation, 
we get Figure \ref{E7:conj:c:fps}, which provides a list of the representatives of the conjugacy classes of $E_7$, the centralizers of each representative, and the corresponding fixed point spaces.

\begin{figure}
\begin{center}
\begin{tabular}{|c | c | c | } 
 \hline Representatives $\beta$  &Centralizers &Fixed point spaces\\
 of Conjugacy classes & $C_{E_7}(\beta)$ & $(S^4)^{\beta}$
 \\ \hline
$1$ & $ E_7$ & $S^4$\\
$-1$ & $E_7$ & $ S^0$\\
$i= t^2$ & $\langle  t \rangle \cong \mathbb{Z}/8$ & $S^0$\\
$s$ & $\langle s\rangle \cong \mathbb{Z}/6$ & $S^0$\\
$-s = s^4$ & $\langle s\rangle \cong \mathbb{Z}/6$ & $S^0$ \\
$r$ & $\langle r\rangle \cong \mathbb{Z}/4$ & $S^0$ \\
$t$ & $\langle t\rangle \cong \mathbb{Z}/8$  & $S^0$ \\
$-t = t^5$ & $\langle t\rangle \cong \mathbb{Z}/8$  & $S^0$ \\
 \hline
\end{tabular} \caption{Conjugacy classes, centralizers and fixed point spaces} \label{E7:conj:c:fps}
\end{center} \end{figure}

Below we give the factor of $QEll_{E_7}(S^4)$ corresponding to each conjugacy class.
\begin{enumerate}
    \item For the conjugacy class of $1$,
    \begin{align*}K_{\Lambda_{E_7}(1)}((S^4)^1) &\cong K_{E_7\times \mathbb{T}} (S^4) \cong K_{E_7}(S^4)\otimes R\mathbb{T} \\
    &\cong K_{E_7}(S^0) \otimes \mathbb{Z}[q^{\pm}] \cong (RE_7 \oplus RE_7)\otimes \mathbb{Z}[q^{\pm}].\end{align*}

\item Then we consider  the conjugacy class of $-1$.

There is a  commutative diagram with each horizontal sequence exact.
\begin{equation}
\xymatrix{ 0\ar[r] &\mathbb{Z}/2 \ar[r] \ar@{=}[d] &E_7 \ar[r]^{\pi} \ar@{^{(}->}[d] &T_7\ar[r] \ar@{^{(}->}[d] &0 \\ 
0\ar[r] &\mathbb{Z}/2 \ar[r] &Spin(3) \ar[r]^{\pi} &SO(3)\ar[r] &0 }
\end{equation} where 
$T_7$ is the octahedral group. 

Thus, by Lemma \ref{dcld2}, we have
\begin{align*}
K_{\Lambda_{E_7}(-1)} ((S^4)^{-1}) &\cong K_{\Lambda_{T_7}(1)}(S^0) \oplus K^{ [\widetilde{\Lambda_{T_7}(1)}_{\rho}]+*}_{\Lambda_{T_7}(1)}(S^0) \\
&\cong K_{T_7\times \mathbb{T}}(S^0) \oplus K^{ [\widetilde{(T_7\times \mathbb{T})}_{\rho}]+*}_{T_7\times \mathbb{T}}(S^0) \\
&\cong K_{T_7}(S^0)\otimes \Z[q^{\pm}] \oplus K^{ [\widetilde{(T_7)}_{\rho}]+*}_{T_7}(S^0)\otimes \Z[q^{\pm}] \\
&\cong (R(T_7)\oplus R (T_7) \oplus R_{[\widetilde{(T_7)}_{\rho}]} (T_7) \oplus R_{[\widetilde{(T_7)}_{\rho}]} (T_7)) \otimes \Z[q^{\pm}].
\end{align*} where $\rho$ is the sign representation of $\Z/2$.

    \item 

    For the conjugacy class of $i$, 
\begin{align*}
     K_{\Lambda_{E_7}(i)} ((S^4)^i)    &\cong  K_{\Lambda_{\mathbb{Z}/8}(2) }(S^0) \cong R\Lambda_{\mathbb{Z}/8}(2)  \oplus R\Lambda_{\mathbb{Z}/8}(2)  \\
      &\cong   \Z[x, q^{\pm }] /\langle x^8-q^2\rangle \oplus \Z[x, q^{\pm }] /\langle x^8-q^2\rangle.
    \end{align*}

    \item For the conjugacy class of  $s= \frac{1}{2}(1+i+j+k) $,  \begin{align*}
      K_{\Lambda_{E_7}(s)} ((S^4)^s)   &\cong  K_{\Lambda_{\mathbb{Z}/6}(1) }(S^0) \cong R\Lambda_{\mathbb{Z}/6}(1)  \oplus R\Lambda_{\mathbb{Z}/6}(1)  \\
      &\cong    \Z[x, q^{\pm }] /\langle x^6-q\rangle \oplus \Z[x, q^{\pm }] /\langle x^6-q\rangle.
    \end{align*}

\item For the conjugacy class of  $-s= -\frac{1}{2}(1+i+j+k) $,  \begin{align*}
       K_{\Lambda_{E_7}(-s)} ((S^4)^{-s})  &\cong K_{\Lambda_{\mathbb{Z}/6}(4) }(S^0) \cong R\Lambda_{\mathbb{Z}/6}(4) \oplus R\Lambda_{\mathbb{Z}/6}(4)  \\
      &\cong   \Z[x, q^{\pm }] /\langle x^6-q^4\rangle \oplus \Z[x, q^{\pm }] /\langle x^6-q^4\rangle.
    \end{align*}

\item For the conjugacy class of $r=\frac{1}{\sqrt{2}} (i+j)$, \begin{align*}
        K_{\Lambda_{E_7}(r)} ((S^4)^{r}) &\cong  K_{\Lambda_{\mathbb{Z}/4}(1) }(S^0) \cong R\Lambda_{\mathbb{Z}/4}(1) \oplus R\Lambda_{\mathbb{Z}/4}(1)  \\
      &\cong    \Z[x, q^{\pm }] /\langle x^4-q\rangle \oplus \Z[x, q^{\pm }] /\langle x^4-q\rangle.
    \end{align*}

\item For the conjugacy class of $t= \frac{1}{\sqrt{2}}(1+i)$, 
\begin{align*}
     K_{\Lambda_{E_7}(t)} ((S^4)^{t})    &\cong  K_{\Lambda_{\mathbb{Z}/8}(1) }(S^0) \cong R\Lambda_{\mathbb{Z}/8}(1) \oplus R\Lambda_{\mathbb{Z}/8}(1)  \\
      &\cong   \Z[x, q^{\pm }] /\langle x^8-q\rangle \oplus \Z[x, q^{\pm }] /\langle x^8-q\rangle.
    \end{align*}

\item For the conjugacy class of $-t= -\frac{1}{\sqrt{2}}(1+i)$, 
 \begin{align*}
       K_{\Lambda_{E_7}(-t)} ((S^4)^{-t})  &\cong  K_{\Lambda_{\mathbb{Z}/8}(5) }(S^0) \cong R\Lambda_{\mathbb{Z}/8}(5) \oplus R\Lambda_{\mathbb{Z}/8}(5)  \\
      &\cong    \Z[x, q^{\pm }] /\langle x^8-q^5\rangle \oplus \Z[x, q^{\pm }] /\langle x^8-q^5\rangle.
    \end{align*}
    
\end{enumerate}

Thus, in conclusion, \begin{align*}
QEll_{E_7}(S^4)  = &K_{\Lambda_{E_7}(1)}((S^4)^1)  \times  K_{\Lambda_{E_7}(-1)}(S^4)^{-1} \times K_{\Lambda_{E_7}(i)} ((S^4)^i) \\
&\times 
 K_{\Lambda_{E_7}(s)} ((S^4)^s)  \times  K_{\Lambda_{E_7}(-s)} ((S^4)^{-s}) \times K_{\Lambda_{E_7}(r)} ((S^4)^{r}) \\
 &\times    K_{\Lambda_{E_7}(t)} ((S^4)^{t})   \times K_{\Lambda_{E_7}(-t)} ((S^4)^{-t})\\
    \cong &(RE_7 \oplus RE_7)\otimes \mathbb{Z}[q^{\pm}] \\
    &\times (R(T_7)\oplus R (T_7) \oplus R_{[\widetilde{(T_7)}_{\rho}]} (T_7) \oplus R_{[\widetilde{(T_7)}_{\rho}]} (T_7)) \otimes \Z[q^{\pm}]\\
    &\times \Z[x, q^{\pm }] /\langle x^8-q^2\rangle \oplus \Z[x, q^{\pm }] /\langle x^8-q^2\rangle \\
    & \times \Z[x, q^{\pm }] /\langle x^6-q\rangle \oplus \Z[x, q^{\pm }] /\langle x^6-q\rangle \\
    & \times \Z[x, q^{\pm }] /\langle x^6-q^4\rangle \oplus \Z[x, q^{\pm }] /\langle x^6-q^4\rangle \\
    & \times  \Z[x, q^{\pm }] /\langle x^4-q\rangle \oplus \Z[x, q^{\pm }] /\langle x^4-q\rangle \\
    & \times \Z[x, q^{\pm }] /\langle x^8-q\rangle \oplus \Z[x, q^{\pm }] /\langle x^8-q\rangle \\
    & \times \Z[x, q^{\pm }] /\langle x^8-q^5\rangle \oplus \Z[x, q^{\pm }] /\langle x^8-q^5\rangle, 
\end{align*}  where $\rho$ is the sign representation of $\Z/2$.

\end{example}

\begin{example} \label{E8S4}

In this example we compute $QEll_{E_8}(S^4)$, where  $E_8$ is the binary icosahedral group. 
A presentation of this group is $$\langle r, s, t \mid (st)^2= s^3= t^5 = -1.\rangle. $$
The cardinality of $E_8$ is 120.
In this example, 
we use $\tau$ to denote $\frac{1+\sqrt{5}}{2}$ and   $\sigma$ to denote the number $\frac{1-\sqrt{5}}{2}$.

By \cite[page 7635, Table 1]{KocaAKo} and direct computation, 
we obtain a list of the representatives of the conjugacy classes of $E_8$, the centralizers of each representative, and the corresponding fixed point spaces in Figure \ref{E8:conj:c:fps}.

\begin{figure}
\begin{center}
\begin{tabular}{|c | c | c | } 
 \hline Representatives $\xi$  &Centralizers &Fixed point spaces\\
 of Conjugacy classes & $C_{E_8}(\xi)$ & $(S^4)^{\xi}$
 \\ \hline
$1$ & $ E_8$ & $S^4$\\
$-1$ & $E_8$ & $ S^0$\\
$y_3:= \frac{1}{2}(\tau + i + \sigma k)$ & $\langle  y_3 \rangle \cong \mathbb{Z}/10$ & $S^0$\\
$y_4:= y_5^2 = \frac{1}{2}(-\tau + \sigma i-j)$ & $\langle  y_5 \rangle \cong \mathbb{Z}/10$ & $S^0$\\
$y_5:= \frac{1}{2}(\sigma + i + \tau j)$ & $\langle  y_5 \rangle \cong \mathbb{Z}/10$ & $S^0$ \\
$y_6:= y_3^2 = \frac{1}{2}(-\sigma + \tau i -k)$ & $\langle  y_3 \rangle \cong \mathbb{Z}/10 $ & $S^0$ \\
$y_7:= \frac{1}{2}(1 + i +  j +k)$ & $\langle  y_7 \rangle \cong \mathbb{Z}/6$  & $S^0$ \\
$y_8:= y_7^2=\frac{1}{2}(- 1  + i + j + k)$ & $\langle  y_7 \rangle \cong \mathbb{Z}/6$  & $S^0$ \\
$y_9:= i $  &$ \langle  y_9 \rangle \cong \mathbb{Z}/4$  & $S^0$ \\
 \hline
\end{tabular} \caption{Conjugacy classes, centralizers and fixed point spaces} \label{E8:conj:c:fps}
\end{center}
\end{figure}

Then we compute the factor of $QEll_{E_8}(S^4)$ corresponding to each conjugacy class of $E_8$ one by one.

\begin{enumerate}
    \item For the conjugacy class of $1$, we have
    \begin{align*}K_{\Lambda_{E_8}(1)}((S^4)^1) &\cong K_{E_8\times \mathbb{T}} (S^4) \cong K_{E_8}(S^4)\otimes R\mathbb{T} \\
    &\cong K_{E_8}(S^0) \otimes \mathbb{Z}[q^{\pm}] \cong (RE_8 \oplus RE_8)\otimes \mathbb{Z}[q^{\pm}].\end{align*}

\item Then we consider the  conjugacy class of $-1$. 
We have the commutative diagram with both horizontal sequences exact.
\begin{equation}
\xymatrix{ 0\ar[r] &\mathbb{Z}/2 \ar[r] \ar@{=}[d] &E_8 \ar[r]^{\pi} \ar@{^{(}->}[d] &T_8\ar[r] \ar@{^{(}->}[d] &0 \\ 
0\ar[r] &\mathbb{Z}/2 \ar[r] &Spin(3) \ar[r]^{\pi} &SO(3)\ar[r] &0, }
\end{equation} where
$T_8$ is the icosahedral group. 
Thus, by Lemma \ref{dcld2}, we have
\begin{align*}
K_{\Lambda_{E_8}(-1)} ((S^4)^{-1}) &\cong K_{\Lambda_{T_8}(1)}(S^0) \oplus K^{ [\widetilde{\Lambda_{T_8}(1)}_{\rho}]+*}_{\Lambda_{T_8}(1)}(S^0) \\
&\cong K_{T_8\times \mathbb{T}}(S^0) \oplus K^{ [\widetilde{(T_8\times \mathbb{T})}_{\rho}]+*}_{T_8\times \mathbb{T}}(S^0) \\
&\cong K_{T_8}(S^0)\otimes \Z[q^{\pm}] \oplus K^{ [\widetilde{(T_8)}_{\rho}]+*}_{T_8}(S^0)\otimes \Z[q^{\pm}] \\
&\cong (R(T_8)\oplus R (T_8) \oplus R_{[\widetilde{(T_8)}_{\rho}]} (T_8) \oplus R_{[\widetilde{(T_8)}_{\rho}]} (T_8)) \otimes \Z[q^{\pm}].
\end{align*} where $\rho$ is the sign representation of $\Z/2$.

\item 
    For the conjugacy class of \[y_3:= \frac{1}{2}(\tau + i + \sigma k), \]
  \begin{align*}
     K_{\Lambda_{E_8}(y_3)} ((S^4)^{y_3})  &\cong K_{\Lambda_{\mathbb{Z}/10}(1) }(S^0) \cong R\Lambda_{\mathbb{Z}/10}(1)  \oplus R\Lambda_{\mathbb{Z}/10}(1)  \\
      &\cong    \Z[x, q^{\pm }] /\langle x^{10}-q\rangle \oplus \Z[x, q^{\pm }] /\langle x^{10}-q\rangle.
    \end{align*}

\item 
    For the conjugacy class of \[y_4:= \frac{1}{2}(-\tau + \sigma i-j), \] since $y_4= y_5^2$, 
\begin{align*}
K_{\Lambda_{E_8}(y_4)} ((S^4)^{y_4})        & \cong K_{\Lambda_{\mathbb{Z}/10}(2) }(S^0) \cong R\Lambda_{\mathbb{Z}/10}(2)  \oplus R\Lambda_{\mathbb{Z}/10}(2)  \\
     & \cong    \Z[x, q^{\pm }] /\langle x^{10}-q^2\rangle \oplus \Z[x, q^{\pm }] /\langle x^{10}-q^2\rangle.
    \end{align*}

    \item     For the conjugacy class of \[y_5:= \frac{1}{2}(\sigma + i + \tau j),\]
\begin{align*}K_{\Lambda_{E_8}(y_5)} ((S^4)^{y_5}) 
        &\cong  K_{\Lambda_{\mathbb{Z}/10}(1) }(S^0) \cong R\Lambda_{\mathbb{Z}/10}(1)  \oplus R\Lambda_{\mathbb{Z}/10}(1)  \\
      & \cong    \Z[x, q^{\pm }] /\langle x^{10}-q \rangle \oplus \Z[x, q^{\pm }] /\langle x^{10}-q\rangle.
    \end{align*}

  \item      For the conjugacy class of \[y_6:= \frac{1}{2}(-\sigma + \tau i -k),\]  since $y_6= y^2_3$, 
\begin{align*} K_{\Lambda_{E_8}(y_6)} ((S^4)^{y_6}) 
        & \cong K_{\Lambda_{\mathbb{Z}/10}(2) }(S^0) \cong R\Lambda_{\mathbb{Z}/10}(2)  \oplus R\Lambda_{\mathbb{Z}/10}(2)  \\
      &\cong    \Z[x, q^{\pm }] /\langle x^{10}-q^2\rangle \oplus \Z[x, q^{\pm }] /\langle x^{10}-q^2\rangle.
    \end{align*}

    \item     For the conjugacy class of \[y_7:= \frac{1}{2}(1 + i +  j +k), \] 
 \begin{align*}
       K_{\Lambda_{E_8}(y_7)} ((S^4)^{y_7})  &\cong  K_{\Lambda_{\mathbb{Z}/6}(1) }(S^0) \cong R\Lambda_{\mathbb{Z}/6}(1)  \oplus R\Lambda_{\mathbb{Z}/6}(1)  \\
      &\cong   \Z[x, q^{\pm }] /\langle x^{6}-q \rangle \oplus \Z[x, q^{\pm }] /\langle x^{6}-q\rangle.
    \end{align*}

  \item      For the conjugacy class of \[y_8:= \frac{1}{2}(- 1  + i + j + k), \]  since $y_8 = y_7^2$, 
 \begin{align*}
      K_{\Lambda_{E_8}(y_8)} ((S^4)^{y_8})  &\cong K_{\Lambda_{\mathbb{Z}/6}(2) }(S^0) \cong R\Lambda_{\mathbb{Z}/6}(2)  \oplus R\Lambda_{\mathbb{Z}/6}(2)  \\
      &\cong    \Z[x, q^{\pm }] /\langle x^{6}-q^2\rangle \oplus \Z[x, q^{\pm }] /\langle x^{6}-q^2\rangle.
    \end{align*}

    \item     For the conjugacy class of \[y_9:= i,\] 
     \begin{align*}
       K_{\Lambda_{E_8}(y_9)} ((S^4)^{y_9}) &\cong  K_{\Lambda_{\mathbb{Z}/4}(1) }(S^0) \cong R\Lambda_{\mathbb{Z}/4}(1)  \oplus R\Lambda_{\mathbb{Z}/4}(1)  \\
      &\cong    \Z[x, q^{\pm }] /\langle x^{4}-q \rangle \oplus \Z[x, q^{\pm }] /\langle x^{4}-q\rangle.
    \end{align*}

\end{enumerate}

In conclusion, \begin{align*}
    QEll_{E_8}(S^4) =& K_{\Lambda_{E_8}(1)}((S^4)^1) \times K_{\Lambda_{E_8}(-1)}((S^4)^{-1}) \times  K_{\Lambda_{E_8}(y_3)} ((S^4)^{y_3}) \\ 
    &\times K_{\Lambda_{E_8}(y_4)} ((S^4)^{y_4})  \times K_{\Lambda_{E_8}(y_5)} ((S^4)^{y_5})  \times  K_{\Lambda_{E_8}(y_6)} ((S^4)^{y_6}) \\
    &\times K_{\Lambda_{E_8}(y_7)} ((S^4)^{y_7}) \times   K_{\Lambda_{E_8}(y_8)} ((S^4)^{y_8}) \times  K_{\Lambda_{E_8}(y_9)} ((S^4)^{y_9})\\
     \cong &  (RE_8 \oplus RE_8)\otimes \mathbb{Z}[q^{\pm}]  \\
     &\times (R(T_8)\oplus R (T_8) \oplus R_{[\widetilde{(T_8)}_{\rho}]} (T_8) \oplus R_{[\widetilde{(T_8)}_{\rho}]} (T_8)) \otimes \Z[q^{\pm}] \\ &\times 
      \Z[x, q^{\pm }] /\langle x^{10}-q\rangle \oplus \Z[x, q^{\pm }] /\langle x^{10}-q\rangle \\
      &\times  \Z[x, q^{\pm }] /\langle x^{10}-q^2\rangle \oplus \Z[x, q^{\pm }] /\langle x^{10}-q^2\rangle\\
      &\times  \Z[x, q^{\pm }] /\langle x^{10}-q \rangle \oplus \Z[x, q^{\pm }] /\langle x^{10}-q\rangle\\
      &\times  \Z[x, q^{\pm }] /\langle x^{10}-q^2\rangle \oplus \Z[x, q^{\pm }] /\langle x^{10}-q^2\rangle\\
      &\times \Z[x, q^{\pm }] /\langle x^{6}-q \rangle \oplus \Z[x, q^{\pm }] /\langle x^{6}-q\rangle\\
      &\times  \Z[x, q^{\pm }] /\langle x^{6}-q^2\rangle \oplus \Z[x, q^{\pm }] /\langle x^{6}-q^2\rangle\\
      &\times  \Z[x, q^{\pm }] /\langle x^{4}-q \rangle \oplus \Z[x, q^{\pm }] /\langle x^{4}-q\rangle,
\end{align*}  where $\rho$ is the sign representation of $\Z/2$.

\end{example}

\subsection{Twisted Quasi-elliptic cohomology of $4$-spheres acted on by finite subgroups of of $SU(2)$} \label{ex:twisted:qell:s4}

In this subsection we compute the twisted quasi-elliptic cohomology of a space $X$ acted on by a finite subgroup $G$ of $SU(2)$.

By 
 \cite[Section 5]{Epa_Ganter_2017PlatonicAA}, for any finite subgroup $G$ of $SU(2)$, $H^2(BG; U(1))=0$. Thus, for any finite subgroup $G$ of $SU(2)$, the target of the transgression map \[H^3(BG; U(1)) \longrightarrow \prod_{[g]} H^2(BC_{G}(g); U(1))\]  is zero, where $[g]$ goes over all the conjugacy classes in $G$. Note that the subgroup $C_G(g)$ for each $g\in G $ is still a finite subgroup of $SU(2)$.  
Thus, there is only one central extension of $G$ by the circle group  $\mathbb{T}$, which is the cartesian product $G\times \mathbb{T}$.
\[1\longrightarrow \mathbb{T} \longrightarrow G\times \mathbb{T}\longrightarrow G\longrightarrow 1.\]

To avoid confusion, we use the symbol\[\alpha\] to denote the only element $1$ in  each target group \[\prod_{[g]} H^2(BC_{G}(g); U(1))\cong \{1\}.\]

We have \begin{align*}
C^\alpha_G(g) &= C_G(g) \times \mathbb{T}; \\
\Lambda^\alpha_G(g) &= \Lambda_G(g) \times \mathbb{T}.
\end{align*} Then, each factor in the twisted quasi-elliptic cohomology $QEll^{\alpha+\ast}_G(X) $ 
\[K^{\alpha+\ast}_{\Lambda_{G}(g)}(X^g) \cong K^{\ast}_{\Lambda_{G}(g)}(X^g). \] Thus, we have the conclusion
\begin{proposition} \label{twistedQEll:X:SU(2)}
    \begin{equation} QEll^{\alpha+\ast}_G(X) \cong QEll^{\ast}_G(X). \end{equation}
\end{proposition}

So we have the corollary below especially for the twisted quasi-elliptic cohomology of $S^4$ acted on by a finite subgroup $G$ of $SU(2)$.
\begin{corollary}\label{twistedQEll:S4:SU(2)} The twisted quasi-elliptic cohomology of $S^4$ acted on by a finite subgroup $G$ of $SU(2)$ in the way as \eqref{gpact:def} is isomorphic to the untwisted theory, i.e.
    \begin{equation} QEll^{\alpha+\ast}_G(S^4) \cong QEll^{\ast}_G(S^4). \end{equation}
\end{corollary}

\begin{example}

In addition, we can compute the twisted version of the quasi-elliptic cohomology of $S^1$ in Section \ref{Ex:QEll:S1}. Explicitly, 
\begin{itemize}
    \item For $S^1$ acted on by $\Z/N$ via the rotation, \[QEll^{\alpha+ \ast}_{\Z/N}(S^1) \cong  \Z[q^{\pm}],\] for any twist $\alpha \in H^3(B\Z/N; U(1))$.
    \item For $S^1$ acted on by $\Z/2$ via the reflection, \[QEll^{\alpha+ \ast}_{\Z/2}(S^1) \cong (\Z[q^{\pm}, x]/ \langle x^2-q\rangle \oplus \Z[q^{\pm}, x]/\langle x^2-q\rangle )\times \Z[q^{\pm}],\] for any twist $\alpha \in H^3(B\Z/2; U(1))$.
\end{itemize}
    
\end{example}

\appendix

\section{Corollaries of {\'A}ngel-G{\'o}mez-Uribe Decomposition Formula} \label{Cor:decomp:dcl}

In this section, we prove some corollaries of \cite[Theorem 3.6, Corollary 3.7]{ngel2017EquivariantCB} that are applied  in  Section \ref{twistedqell}.
The corollaries all apply to compact Lie groups.

\begin{lemma} 
Let $Q$ and $G$ be compact Lie groups. And we have a short exact sequence \[ 1 \longrightarrow \Z/2  \buildrel{l}\over\longrightarrow G \buildrel{\pi}\over\longrightarrow Q \longrightarrow 1 \] and $l(A)$ is contained in the center of $G$. Let $X$ be a $G$-space with $l(\Z/2)$ acting on it trivially.
Then, we have the isomorphism
\[K^*_{G}(X)\cong K^*_{Q}(X)\oplus K^{[\tilde{Q}_{sign}]+*}_{Q}(X)\]
\label{dcl}
\end{lemma}
\begin{proof}
As given in \cite[Section 2.1]{ngel2017EquivariantCB}, there is a well-defined $G$-action on the irreducible $\Z/2$-representations by \[(g\cdot \rho)(a) = \rho (g^{-1} a g) = \rho (a),\] for any $g\in G$, $a\in \Z/2$ and any irreducible $\Z/2$-representation $\rho$.  

Since the irreducible representations $(\rho, V_{\rho})$ of $\mathbb{Z}/2$ are all 1-dimensional and fixed by $G$,  
 the group $PU(1)$ of inner automorphism of $U(1)$ consists of exactly one element, i.e. the identity map. As in \cite[(1), page 6]{ngel2017EquivariantCB}, we use the symbol $\tilde{G}_{\rho}$ to denote the pullback  
\[ \xymatrix{\tilde{G}_{\rho} \ar[r]^{\tilde{f}} \ar[d]^{\tau_{\rho}} &U(1) \ar[d] \\
G\ar[r] &PU(1)} \] We have $\tilde{G}_{\rho} = G\times U(1)$. The map $\tau_{\rho}$ is the projection map to $G$ and $\tilde{f}$ is the projection map to $U(1)$. 

Then we consider the commutative diagram
\[ \xymatrix{\Z/2 \ar[r]^{\tilde{l}} \ar[d]_{=} &\tilde{G}_{\rho} \ar[d] \\ \Z/2 \ar[r]^{l} &G}\] where $\tilde{l}$ is defined to be the unique map so that $\rho=\tilde{f}\circ \tilde{l}$. Thus, 
$\tilde{l}$ is the product of $l$ and the representation $\rho$. 

Then
we consider the commutative diagram \begin{equation} \label{qrho:plbk}
    \xymatrix{  &\Z/2 \ar[d]^{\tilde{l}} & \Z/2\ar[d]^{l} \\ \T \ar[r] & \tilde{G}_{\rho} \ar[d]^{\tilde{\pi}} \ar[r] &G \ar[d]^{\pi} \\
\T\ar[r]^{i_Q} & \tilde{Q}_{\rho} \ar[r]^{p_Q} & Q   }
\end{equation} where the vertical sequences are both exact, the horizontal sequences are $\T$-central extensions and the square is a pullback square. 
If $\rho$ is the trivial representation of $\Z/2$, $\tilde{Q}_{\rho} \cong Q\times \T$ and, by \cite[Proposition 2.2]{ngel2017EquivariantCB}, $\rho$ extends to an irreducible representation of $G$. However, if $\rho$ is the sign representation of $\Z/2$,  it may not extend to the whole group $G$. And the central extension \[ \xymatrix{1\ar[r] & \T \ar[r]^{i_Q} & \tilde{Q}_{\rho}  \ar[r]^{p_Q}
& Q \ar[r] &1}\] may correspond to a nontrivial element $[\tilde{Q}_{\rho}]$  in $H^3(BQ; \mathbb{Z})$.




By \cite[Corollary 3.7]{ngel2017EquivariantCB},
\begin{equation}
K^*_{G}(X) \cong \bigoplus_{\rho\in G \slash Irr(\Z/2)} K^{[\tilde{Q}_{\rho}] +\ast}_{Q_{\rho}} (X), \label{decomp:z2:cor}\end{equation}
where $\rho$ runs over representatives of the orbits of the $G$-action on the set of isomorphism classes of irreducible $\Z/2$-representations, i.e. $\{1, sign\}$, the action of \[Q_{\rho} = G_{\rho}/(\Z/2)\] on $X$ is induced from the $G$-action on $X$, and $G_{\rho}$ is the isotropy group of $\rho$ under the $G$-action. Note that the two irreducible $\Z/2$-representations  are fixed by the $G$-action and $G_{\rho} = G$ for each $\rho$. Thus, the isomorphism \eqref{decomp:z2:cor} is exactly 
\[K^*_{G}(X)\cong K^*_{Q}(X)\oplus K^{[\tilde{Q}_{sign}]+*}_{Q}(X)\] In each component, the $Q$-action on $X$ is induced from the quotient map $\pi: G\longrightarrow Q$.

\end{proof}

Let
\[ 1 \longrightarrow \Z/2  \buildrel{l}\over\longrightarrow G \buildrel{\pi}\over\longrightarrow Q \longrightarrow 1 \] be a short exact sequence of compact groups and $l(A)$ is contained in the center of $G$.
For any torsion element $\alpha$ in $G$, we have the short exact sequence
\[0\longrightarrow \mathbb{Z}/2\buildrel{i}\over\longrightarrow \Lambda_{G}(\alpha)\buildrel{[\pi, id]}\over\longrightarrow \Lambda_{Q}(\pi(\alpha))\longrightarrow 0\] with \[i(\Z/2)=\{[\beta, 0] \in \Lambda_G(\alpha) \mid \beta\in l(\Z/2)\}  \] contained in the center of $\Lambda_{G}(\pi(\alpha))$. 
In addition, $X^\alpha$ is a $ \Lambda_{G}(\alpha)$-space with the action  by 
$i(\Z/2)$ trivial.

Especially, if $\alpha$ is the nontrivial element in $l(\Z/2)$, then $\pi(\alpha) = 1$ and we have \[\Lambda_{Q}(\pi(\alpha))\cong Q\times \T; \quad
\widetilde{\Lambda_{Q}(\pi(\alpha))}_{\rho}  \cong  \tilde{Q}_{\rho}\times \T. \]
In this case, the central extension  \[ \xymatrix{1\ar[r] & \T \ar[r] & \widetilde{\Lambda_{Q}(\pi(\alpha))}_{\rho}  \ar[r]
& \Lambda_{Q}(\pi(\alpha)) \ar[r] &1} \] is completely determined by \[ \xymatrix{1\ar[r] & \T \ar[r]^{i_Q} & \tilde{Q}_{\rho}  \ar[r]^{p_Q}
& Q\ar[r] &1}, \] thus, by the 3-cocycle $[\tilde{Q}_{\rho}]$.

Then we can get a corollary of Lemma \ref{dcl}.
\begin{lemma} Let
\[ 1 \longrightarrow \Z/2  \buildrel{l}\over\longrightarrow G \buildrel{\pi}\over\longrightarrow Q \longrightarrow 1 \] be a short exact sequence of compact groups and  $l(A)$ is contained in the center of $G$. Let $X$ be a $G$-space with $l(\Z/2)$ acting on
it trivially. For any torsion element $\alpha$ in $G$,
we have the isomorphism
\[ K^*_{\Lambda_G(\alpha)}(X^{\alpha}) \cong K^*_{\Lambda_{Q}(\pi(\alpha))}(X^{\alpha})\oplus K^{[\widetilde{\Lambda_{Q}(\pi(\alpha))}_{sign}]+*}_{\Lambda_{Q}(\pi(\alpha))}(X^{\alpha}).\]

Especially,  if $\alpha$ is the nontrivial element in $l(\Z/2)$, \[K^*_{\Lambda_G(\alpha)}(X^{\alpha}) \cong K^*_{Q}(X^{\alpha})\otimes \Z[q^{\pm}]\oplus K^{[\tilde{Q}_{sign}]+*}_{Q}(X^{\alpha}) \otimes \Z[q^{\pm}]. \]

\label{dcld2}
\end{lemma}

\section*{Declaration}
The author declares that she has no conflicts of interest.

\bibliographystyle{amsalpha}
\bibliography{twistedQEll.bib}

\end{document}